\documentclass[a4paper]{article}

\usepackage[utf8]{inputenc}
\usepackage[T1]{fontenc}
\usepackage[english]{babel}
\usepackage{hyperref}
\usepackage{caption}
\usepackage{subcaption}
\usepackage{mathtools}
\usepackage{float}
\usepackage{tikz}
\usetikzlibrary{patterns,decorations.pathmorphing,positioning,arrows,fit,decorations.pathreplacing,calc,shapes.arrows,decorations.markings,shapes.multipart}
\usepackage{array}
\usepackage{authblk}

\makeatletter
\renewcommand\section{\@startsection
	{section}
{1}
{0pt}
{-3.5ex plus -1ex minus -.2ex}
{2.3ex plus.2ex}
{\centering\normalfont\Large\scshape}}

\renewcommand\subsection{\@startsection
	{subsection}
{2}
{0pt}
{-3ex plus -1ex minus -.2ex}
{1ex plus.2ex}
{\normalfont\large\bfseries}}

\renewcommand\subsubsection{\@startsection
	{subsubsection}
{3}
{0pt}
{-1.5ex plus -1ex minus -.2ex}
{0.8ex plus .2ex}
{\normalfont\bfseries}}

\renewcommand\paragraph{\@startsection
	{paragraph}
{4}
{0em}
{-1.2ex plus -0.4ex minus -.2ex}
{0ex}
{\normalfont\bfseries }}

\makeatother

\usepackage{latexsym,amsfonts,amsmath,amssymb,amsthm,amscd}
\usepackage{thmtools}
\usepackage{thm-restate}

\usepackage{bbm}
\usepackage{yhmath}

\usepackage[capitalise]{cleveref}

\usepackage{subcaption}

\usepackage{enumerate}

\newcommand{\lang}{\mathcal{L}}

\newcommand{\Z}{\mathbbm{Z}}
\newcommand{\N}{\mathbbm{N}}
\newcommand{\M}{\mathcal{M}}

\newcommand{\A}{\mathcal{A}}
\newcommand{\B}{\mathcal{B}}

\newcommand{\AZ}{\mathcal{A}^\mathbb{Z}}
\newcommand{\gls}{\widetilde\omega} 
\newcommand{\I}{\mathcal{I}} 
\newcommand{\W}{\mathcal{W}} 
\newcommand{\ie}{\textit{i}.\textit{e}.} 
\newcommand{\sg}{\sigma}    
\newcommand{\realm}{\mathcal{D}} 

\DeclareMathOperator{\poly}{poly}

\theoremstyle{plain}
\newtheorem{theorem}{Theorem}[section]
\newtheorem{proposition}[theorem]{Proposition}

\newtheorem{corollary}[theorem]{Corollary}
\newtheorem{lemma}[theorem]{Lemma}

\theoremstyle{definition}
\newtheorem*{definition}{Definition}

\theoremstyle{remark}

\newtheorem{ask}{Open problem}[section]

\newcounter{claimcount}[theorem]

\newcommand{\bprf}[1][Proof:]{\begin{list}{}    {\setlength{\leftmargin}{0.5em}
\setlength{\rightmargin}{0em}  \setlength{\listparindent}{1em}}   \item {\em
\hspace{-1em}  #1  }}
\newcommand{\eprf}{\end{list}}

\usepackage{marvosym,fancybox}

\title{Arithmetical Complexity of the Language of Generic Limit Sets of Cellular Automata}
\date{}

\author[1]{Solène J. Esnay}
\author[1]{Alonso N\'u\~nez\footnote{Alonso N\'u\~nez is supported by the National Agency for Research and Development (ANID) /  Scholarship Program / DOCTORADO BECAS CHILE/2019 - 72200562.}}
\author[2]{Ilkka Törmä\footnote{Ilkka Törmä was supported by a grant from Magnus Ehrnrooth Foundation.}}
\affil[1]{Institut de Mathématiques de Toulouse\\
  Universit\'e Paul Sabatier\\
  Toulouse, France}
\affil[2]{Department of Mathematics and Statistics\\
  University of Turku\\
  Turku, Finland}

\begin{document}
\maketitle

\begin{abstract}
	The generic limit set of a dynamical system is the smallest set that attracts most of the space in a topological sense: it is the smallest closed set with a comeager basin of attraction. Introduced by Milnor, it has been studied in the context of one-dimensional cellular automata by Djenaoui and Guillon, Delacourt, and Törmä. In this article we present complexity bounds on realizations of generic limit sets of cellular automata with prescribed properties. 
	We show that generic limit sets have a $\Pi^0_2$ language if they are inclusion-minimal, a $\Sigma^0_1$ language if the cellular automaton has equicontinuous points, and that these bounds are tight. We also prove that many chain mixing $\Pi^0_2$ subshifts and all chain mixing $\Delta^0_2$ subshifts are realizable as generic limit sets. As a corollary, we characterize the minimal subshifts that occur as generic limit sets.
	
\end{abstract}

\section{Introduction}

Introduced in the 40's by Ulam and von Neumann, one-dimensional cellular automata (CA) are discrete dynamical systems where the ambient space is the set $\AZ$ of bi-infinite sequences, and the action is given by a local rule which is applied synchronously on each cell. They are both simply describable and behaviorally rich dynamical systems.

More accessible than the description of local trajectories, attractors of a dynamical system are sets that aim to capture, in some sense, its asymptotic behavior. They constitute a powerful tool to understand and describe a system. Even though the notion of attractor seems fairly intuitive, several non-equivalent definitions can be found in the literature \cite{milnor}. The first introduced and most studied of them is the limit set, made of the configurations that appear infinitely often in the system as a whole over time. Limit sets have been widely studied in the context of cellular automata \cite{ballier2011limit,kari1992nilpotency,culik1989limit}.

Other attractors exist, such as the likely limit set and the generic limit set, both introduced in \cite{milnor}: they are the smallest closed sets that attract ``most of the space'', where the notion of ``most'' is either measure-theoretical (full measure) or topological (comeager). A recent article by Djenaoui and Guillon about the generic limit sets of cellular automata \cite{djenaouiguillon} allows for their combinatorial characterization.
Most notably, generic limit sets are subshifts, another structurally rich kind of discrete dynamical systems. Understanding which of these subshifts can be realized as attractors, depending on the constraints put on the cellular automaton at their base, strengthens a very deep link between two much-studied dynamical systems.

Subsequent research on generic limit sets includes several bounds on the complexity of their language and other related decision problems.
Notably, T\"orm\"a proved in \cite{Ilkka} that the language is at most $\Sigma^0_3$ and that the bound is tight.
That article also presents constraints on the dynamical structure of generic limit sets, some of which we use in this paper.
In \cite{De21}, Delacourt proved a version of Rice's theorem for generic limit sets: all of their nontrivial properties are undecidable.
In \cite{To21}, T\"orm\"a characterized cellular automata whose generic limit set is a singleton and showed that this property is $\Sigma^0_2$-complete.
Similar complexity bounds are currently being investigated in the more general framework of dynamical systems \cite{rojassablik}.

Another notion of attractor, the $\mu$-limit set \cite{kurkamaass}, consists of the configurations made of the words that keep appearing with positive probability as time goes to infinity. Recent results \cite{mulimit,boyer2015mu} prove that different constraints on the base cellular automaton result in $\mu$-limit sets with different arithmetical complexities. In this article, we take the tools first developed in \cite{constructionfirst} and later used in \cite{mulimit,boyer2015mu} to construct $\mu$-limit sets, and adapt them to generic limit sets, as was done in \cite{De21,To21}. The main tool is the ``walls-and-counters'' construction of cellular automata with the property that almost all (both in the measure-theoretic and topological sense) initial configurations are divided into segments of finite length separated by walls, and the contents of the segments can be controlled exactly.

The adaptation process is not trivial: $\mu$-limit sets are in some sense nicer attractors than generic limit sets, because as long as an auxiliary state or pattern used as part of the construction occurs with frequency decreasing to zero in a randomly chosen trajectory, it will not be visible in the $\mu$-limit set.
By comparison, in order for the pattern to not occur in the generic limit set, it must eventually vanish and never reappear.
More precisely, every initial configuration must admit a small perturbation that does not produce the offending pattern after some finite number of time steps, even if perturbed again (by a suitable smaller amount).

In this article, we find new bounds on the complexity of the language of generic limit sets under some structural constraints. We show that the language of an inclusion-minimal generic limit set (one that does not properly contain another topological attractor) is $\Pi^0_2$, that it is $\Sigma^0_1$ if the cellular automaton has equicontinuity points, and that both bounds are tight.
We also prove that all shift-minimal generic limit sets are inclusion-minimal, so the former bound applies to them.

The tightness of the first bound follows from a more general realization result: we show that every chain mixing $\Pi^0_2$ subshift that contains a nonempty $\Pi^0_1$ subshift occurs as a generic limit set of a CA that acts as a shift on it, and the same holds for chain mixing $\Delta^0_2$ subshifts.
Such generic limit sets are inclusion-minimal by the results of \cite{Ilkka}, and thus $\Pi^0_2$-complete inclusion-minimal generic limit sets can be effectively built.
As corollaries, we prove that the chain mixing condition in the realization result cannot be weakened to chain transitive, and we characterize the generic limit sets of cellular automata among one-dimensional transitive SFTs (they are exactly the mixing ones) and minimal subshifts (they are exactly the chain mixing $\Delta^0_2$ ones).

This article is laid out as follows.
\Cref{sec:definitions} is dedicated to definitions.
\Cref{sec:lemmas} presents known and auxiliary results.
In \cref{sec:complex-obstr} we prove constraints on the complexity and structure of generic limit sets, including the aforementioned $\Pi^0_2$ and $\Sigma^0_1$ bounds.
\Cref{sec:gener-constr} presents the basic walls-and-counters construction used in \cite{mulimit,boyer2015mu}, which we adapt in \cref{sec:realizations} to prove the tightness of the $\Sigma^0_1$ bound and in \cref{sec:structure} for the realization result.
Finally, \cref{sec:future} is dedicated to some final remarks and discussion on open questions and future research.


\section{Definitions}
\label{sec:definitions}

Let $X$ be a compact metric space with metric $d$. A subset $Y \subset X$ is \emph{meager} (or \emph{of first category}) if it is the union of countably many sets whose closure has empty interior. It is \emph{comeager} (or \emph{residual}) if its complement is meager.
By the Baire Category Theorem, in our setting all comeager sets are dense in $X$.
If $U \subset X$ is open and $Y \cap U$ is comeager in the relative topology on $U$, we say $Y$ is \emph{comeager in $U$}.
We say $Y$ has the \emph{Baire property} if there is an open set $U \subset X$ such that the symmetric difference $(Y \setminus U) \cup (U \setminus Y)$ is meager.
Subsets of $X$ with the Baire property form a $\sigma$-algebra, so in particular every Borel set has the Baire property.
If $Y$ has the Baire property, then it is nonmeager if and only if it is comeager in every nonempty open $U \subset X$.
See \cite[Section 8]{Ke95} for an overview of Baire category.

For $f\colon X \rightarrow X$ a continuous function, $(X,f)$ forms a \emph{dynamical system}.
A point $x \in X$ is an \emph{equicontinuity point} if
\[
\forall y \in X, \forall \epsilon > 0, \exists \delta > 0, d(x,y) < \delta \Rightarrow \forall n \in \N, d(f^n(x),f^n(y))  < \epsilon.
\]
A \emph{morphism} $h$ between two dynamical systems $(X,f)$ and $(Y,g)$ is a continuous function with $h \circ f = g \circ h$.
If $f$ is surjective, it is called a \emph{factor map}.

The \emph{omega limit}, $\omega(x)$, of $x \in X$ is the collection of all accumulation points in $X$ of the orbit of $x$, that is $\omega(x) := \bigcap_{N\in\N}\overline{\{f^n(x) \mid n \geq N\}}$. Given a closed set $A\subset X$, we define the \emph{basin} (or \emph{realm}) \emph{of attraction} $\realm(A)$ of $A$ by $\realm(A) := \{x\in X:\omega(x)\subset A\}$.
This set has the Baire property \cite[proof of Prop. 3.12]{djenaouiguillon}.
We say $A$ is a (topological) \emph{attractor} if $\realm(A)$ is nonmeager, and a (topological) \emph{generic attractor} if $\realm(A)$ comeager. An attractor $A$ is \emph{inclusion-minimal} if $\realm(B)$ is meager for every closed set $B \subsetneq A$, or in other words, $A$ does not properly contain another attractor. The intersection of all generic attractors, that is, the smallest generic attractor, is called the \emph{generic limit set} (GLS for short) and denoted $\gls(f)$ \cite[Appendix 1]{milnor}.

For a finite set $\A$ called the \emph{alphabet}, we denote by $\A^* = \cup_{n \in \N_0} \A^n$ the set of all \emph{words} over $\A$. They are written $u = u_0 \dots u_{n-1}$ with each $u_i \in \A$, where $|u|:=n$ is the length of $u$.
We also denote $\A^{\leq n} = \bigcup_{k \leq n} \A^k$.
The \emph{full shift} $\A^\Z$ is the set of all two-way infinite \emph{configurations} over $\A$.
A word or configuration $x \in \A^* \cup \A^\Z$ is \emph{periodic with period $p \geq 1$}, if $x_i = x_{i+p}$ holds whenever both values are defined.
For $i \leq j$, we define the \emph{subword} $x_{[i,j]} = x_i x_{i+1}\dots x_{j-1} x_j$. We write $v \sqsubset x$ if $v$ is a subword of $x$.

Given a word $u\in\A^*$, we define the \emph{cylinder} $[u]_i:=\{x\in\A^\Z \mid x_{[i,i+|u|-1]}=u\}$, with $[u] := [u]_0$. Cylinders form a basis of clopen sets for the prodiscrete topology on the space $\A^\Z$, which is compact.
This topology is also induced by the metric $d(x,y) := \inf \{ 2^{-n} \mid x_{[-n,n]} = y_{[-n,n]} \}$.

The \emph{(left) shift map} $\sg : \A^\Z \to \A^\Z$ is the homeomorphism defined by $\sg(x)_i=x_{i+1}$.
This makes $(\A^\Z, \sg)$ a dynamical system.
A \emph{subshift} $X$ is a $\sg$-invariant closed subset of $\AZ$.
If $X$ does not properly contain another nonempty subshift, we say $X$ is \emph{shift-minimal} (or \emph{minimal} if no confusion arises with inclusion-minimality, since the latter property makes sense in the context of attractors only). A subshift $X$ can be described by a set $F\subset\A^*$ of \emph{forbidden words} via $X=\AZ\setminus\left(\bigcup_{u\in F}\bigcup_{i\in\Z}[u]_i\right)$. If the set $F$ can be chosen finite, then $X$ is called a \emph{subshift of finite type} (SFT).
If $F \subset \A^{\leq n}$, we say $X$ has \emph{window size} $n$.

The \emph{language} $\lang(X) = \{v \in \A^* \mid x \in X, v \sqsubset x\}$ of a subshift $X$ is the (countable) set of all subwords of configurations in $X$, and $\lang_n(X) = \lang(X) \cap \A^n$ is the set of subwords of length $n$. We say $X$ is \emph{transitive} if for all $u, v \in \lang(X)$, there exists $w \in \lang(X)$ so that $u w v \in \lang(X)$; it is \emph{mixing} if furthermore for any large enough $n \in \N$ (possibly depending on $u$ and $v$), we can find such a $w$ of length $n$. If $X$ is a mixing SFT, then there exists $n \in \N$, called its \emph{mixing distance}, that works for all $u, v$.

The \emph{SFT approximation of width $n$} of $X$ is the SFT $\mathcal{S}_n(X)$ defined by the forbidden patterns $\A^n \setminus \lang_n(X)$. We say that $X$ is \emph{chain transitive} (resp. \emph{chain mixing}) if every $\mathcal{S}_n(X)$ is transitive (resp. mixing).
See \cite[pp. 66 and 175]{Akin} for definitions in the context of general dynamical systems, equivalent to these ones in the special case of subshifts. Alternatively, see \cite[Corollary 4]{Ka08}: being chain transitive (resp. chain mixing) is being transitive (resp. mixing) for any pair of words \emph{of the same length}.

A (one-dimensional) \emph{cellular automaton} (CA) is a pair $(\A,f)$ where $\A$ is an alphabet and $F\colon \A^{2r+1} \rightarrow \A$ is a \emph{local rule} of \emph{radius} $r \in \N$.
It defines a \emph{global rule} $f \colon \A^\Z \rightarrow \A^\Z$ by $f(x)_i = F(x_{i-r}, x_{i-r+1}, \ldots, x_{i+r})$ for all $x \in \A^\Z$ and $i \in \Z$. Alternatively, $f$ is an endomorphism of the dynamical system $(\A^\Z, \sigma)$. With this, $(\A^\Z,f)$ forms a dynamical system.

A word $b \in \A^{2k+1}$ is a \emph{blocking word} for the CA $f$ of radius $r$ if there exists a sequence of words $v_n \in \A^r$ such that for any $x \in [b]_{-k}$, we have $f^n(x) \in [v_n]$ for all $n \in \N$.
Note how an occurrence of a blocking word in a configuration completely disconnects coordinates to its right and left for the action of the automaton, and that $b$ can be a blocking word without any $v_n$ being one. 
By \cite[Prop. 2.1]{equicontinuity}, a one-dimensional CA admits an equicontinuity point if and only if it admits a blocking word, and furthermore, the sequence $(v_n)_{n \in \N}$ is then eventually periodic.

A \emph{Turing Machine} is a 5-tuple $(Q, \Gamma, q_i, q_f, \delta)$, where $Q$ is a finite set of states with initial state $q_i$ and final state $q_f$, $\Gamma$ is an alphabet and $\delta$ is a transition rule: the machine starts in state $q_i$ and is represented by a read/write head on an infinite discrete tape. It reads letters written with $\Gamma$ on the tape, starting on a given (possibly empty) input in $\Gamma^*$, and overwrites them according to the transition rule $\delta$ and its current state $q \in Q$. If the machine reaches the state $q_f$, it \emph{halts} and outputs the content of its tape.

A predicate over $\N$ describable by a Turing Machine -- that is, that can be the output of a Turing Machine starting on the empty input -- is \emph{computable}, also denoted as $\Sigma_0^0$ and $\Pi_0^0$. If $\phi$ is $\Pi_n^0$ then we say that $\exists k_1 \dots \exists k_m \phi$ is a $\Sigma_{n+1}^0$ predicate, and if $\phi$ is $\Sigma_n^0$ then $\forall k_1 \dots \forall k_m \phi$ is $\Pi_{n+1}^0$. We also define $\Delta_n^0=\Pi_n^0 \cap \Sigma_n^0$. We extend the notations $\Pi_n^0, \Sigma_n^0$ and $\Delta_n^0$ to countable sets described by such predicates: this is called the \emph{arithmetical hierarchy}.

Given two $\Sigma_n^0$ (resp. $\Pi_n^0$) sets $A$ and $B$, we say that $B$ can be \emph{reduced} to $A$ if provided with an enumeration of $A$, we can enumerate $B$: the cost of algorithmically describing $B$ is at most the cost of describing $A$. 
A set is \emph{$\Sigma_n$-hard} if any $\Sigma_n$ set can be reduced to $A$; it is \emph{$\Sigma_n$-complete} if it is $\Sigma_n$ and $\Sigma_n$-hard. Similar definitions hold for $\Pi_n$ and $\Delta_n$. 
The language of a subshift, being countable, can be given a classification in the arithmetical hierarchy; for simplicity, we say the subshift itself has that classification.

\section{Preliminary results}
\label{sec:lemmas}

In this section we present known and auxiliary results on the generic limit sets of cellular automata.
Some of them may hold in greater generality, but we state them only within our context, for simplicity.

\begin{lemma}[Prop. 4.11 \cite{djenaouiguillon}]
  \label{lemma1}
  Let $f$ be a CA. Then the generic limit set $\gls(f)$ is a nonempty $f$-invariant subshift.
\end{lemma}

\begin{lemma}[Lemma 2 in \cite{Ilkka}]
  \label{lemma2}
  Let $f$ be a CA on $\A^\Z$. A word $s \in \A^*$ occurs in $\gls(f)$ if and only if there exists a word $v \in \A^*$ and $i \in \Z$ such that for all $u, w \in \A^*$ there exist infinitely many $t \in \N$ with $f^t([uvw]_{i-|u|}) \cap [s] \neq \emptyset$.
  
  We say that $v$ \emph{enables} $s$ for $f$.
\end{lemma}

\begin{lemma}[Lemma 3 in \cite{Ilkka}]
  \label{lemma3}
  Let $f$ be a CA on $\A^\Z$, let $n \in \N$ and let $[v]_i \subset \A^\Z$ be a cylinder set. Then there exists a cylinder set $[w]_j \subset [v]_i$ and $T \in \N$ such that for all $t \geq T$ we have $f^t([w]_j) \subset [\mathcal{L}_n(\gls(f))]$.
  
  We say that $w$ is \emph{$\gls(f)$-forcing}.
\end{lemma}

\begin{lemma}
  \label{lemma4}
  Let $f$ be a CA on $\A^\Z$ with generic limit set $\gls(f)$.
  The following conditions are equivalent:
  \begin{enumerate}
  \item
    $\gls(f)$ is inclusion-minimal.
  \item
    \label{it:E}
    For all $s \in \lang(\gls(f))$, $v \in \A^*$ and $i \in \Z$, there are infinitely many $t \in \N$ with $f^t([v]_{i}) \cap [s] \neq \emptyset$.
  \end{enumerate}
\end{lemma}

Note that if $s \notin \lang(\gls(f))$, then the $v$ and $i$ described in \cref{it:E} cannot exist due to \cref{lemma2}.
The lemma characterizes the situation in which all choices of $v$ and $i$ are valid whenever one is.
An equivalent formulation of \cref{it:E} is that $\bigcup_{t \geq T} f^{-t}([s])$ is dense in $\A^\Z$ for all $s \in \lang(\gls(f))$ and $T \in \N$.

\begin{proof}
  Suppose that \cref{it:E} holds.
  For each word $s \in \lang(\gls(f))$ and $n \in \N$, the set $\bigcup_{t \geq T} f^{-t}([s]_n) = \sigma^{-n}(\bigcup_{t \geq T} f^{-t}([s]))$ is open, and dense by assumption, hence the intersection $B(s,n) = \bigcap_{T \in \N} \bigcup_{t \geq T} f^{-t}([s]_n)$ is comeager.
  Then $B = \bigcap_{s \in \lang(\gls(f))} \bigcap_{n \in \N} B(s,n) = \{ x \in \A^\Z \mid \gls(f) \subset \omega(x) \}$ is comeager as well, since any language is countable as a subset of $\A^*$.
  Consider any closed set $K \subset \A^\Z$.
  If there exists $x \in \realm(K) \cap B$, then $\gls(f) \subset \omega(x) \subset K$, so $K$ is not a proper subset of $\gls(f)$.
  Otherwise $\realm(K) \subset \A^\Z \setminus B$ is meager.
  This means $\gls(f)$ is inclusion-minimal.
  
  Suppose then that \cref{it:E} does not hold: there exist $s \in \lang(\gls(f))$, $v \in \A^*$, $i \in \Z$ and $T \in \N$ such that $f^t([v]_i)$ does not intersect $[s]$ for any $t \geq T$.
  Let $K = \gls(f) \setminus [s]$, a closed proper subset of $\gls(f)$.
  Then the realm $\realm(K)$ contains $\realm(\gls(f)) \cap [v]_i$, which is nonmeager as the intersection of a comeager set and an open set.
  Hence $\gls(f)$ is not inclusion-minimal.
\end{proof}

Say that a CA $f\colon \A^\Z \to \A^\Z$ is \emph{eventually oblique} on an $f$-invariant subshift $X \subset \A^\Z$ if there exists $n \in \N$ such that $f^n|_X$ admits a neighborhood that is contained in either $(-\infty, -1]$ or $[1, \infty)$.
The proof of~\cite[Proposition~4]{Ilkka} shows that if $f$ is eventually oblique on $\gls(f)$, then it satisfies \cref{it:E} of \cref{lemma4}.
Hence we have the following.

\begin{corollary}
  \label{cor:ShiftGivesMinimal}
  If a CA $f\colon \A^\Z \to \A^\Z$ is eventually oblique on $\gls(f)$, then $\gls(f)$ is inclusion-minimal.
  In particular, this holds if the restriction of $f$ to $\gls(f)$ is a nontrivial shift map.
\end{corollary}

As an aside, we show that even though the generic limit set of a CA might properly contain closed sets with nonmeager realms, these sets cannot be subshifts.
In fact, we can characterize the generic limit set as the smallest subshift with a nonmeager realm.

\begin{proposition}
  Let $f$ be a CA on $\A^\Z$ and $X \subset \A^\Z$ a subshift.
  If $\realm(X)$ is nonmeager, then $\gls(f) \subset X$.
\end{proposition}

\begin{proof}
  Suppose for a contradiction that $\gls(f) \setminus X \neq \emptyset$.
  Then there exists a word $s \in \lang(\gls(f)) \setminus \lang(X)$.
  Let $v \in \A^*$ and $i \in \Z$ be given by \cref{lemma2} applied to $s$.
  Similarly to the proof of \cref{lemma4}, the set $B(s) = \bigcap_{T \in \N} \bigcup_{t \geq T} f^{-t}([s])$ is comeager in $[v]_{i}$.

  As $X$ is closed, its realm $\realm(X)$ has the Baire property, and since $\realm(X)$ is by assumption nonmeager, it is comeager in some nonempty open set $U \subset \A^\Z$.
  Moreover, since the shift map $\sigma$ is a homeomorphism and commutes with $f$, for any $x \in \realm(X)$, $\sigma(x)$ is so that $\omega(\sigma(x)) = \sigma(\omega(x)) \subset X$. As such, $\realm(X)$ is stable by $\sigma$. Up to considering some $\sigma^k(U), k \in \Z$ in which $\realm(X)$ is also comeager instead of $U$, we can assume that $V = [v]_{i} \cap U \neq \emptyset$.
  Hence $\realm(X) \cap B(s)$ is comeager in $V$, in particular nonempty.
  Any configuration $x$ in this set satisfies $\omega(x) \subset X$ and $\omega(x) \cap [s] \neq \emptyset$, so $X$ intersects $[s]$.
  This contradicts $s \notin \lang(X)$.
\end{proof}

%
%

\section{Obstructions}
\label{sec:complex-obstr}

Several bounds in complexity for generic limit sets were already known from \cite{Ilkka}; we mention the following, to give some perspective to the next results:

\begin{proposition}[\cite{Ilkka}, Th. 1]
	The language of the generic limit set of any CA is $\Sigma_3^0$, and there exists a CA with a $\Sigma_3^0$-complete GLS, making the complexity bound tight.
\end{proposition}

\begin{proposition}[\cite{Ilkka}, Prop. 1]
	If the generic limit set $\gls(f)$ of a given CA is a shift-minimal subshift, then its language is $\Sigma^0_2$.
\end{proposition}

We also deduce the following corollary from \cite[Prop. 6]{Ilkka}:

\begin{corollary}
  \label{coro:chaintrans}
  If a subshift is chain transitive and has a finite factor that does not consist of fixed points, then it is not the generic limit set of any CA.
\end{corollary}

One of our main results, \cref{thm:MixingRealization}, concerns cellular automata that act as the shift map on their generic limit set: we realize a class of chain mixing subshifts as such generic limit sets.
We now show that the chain mixing assumption is necessary in this context.

\begin{lemma}[\cite{Akin}, p. 175]
  \label{lem:Akin}
  If $(X,T)$ is a chain transitive topological dynamical system that is not chain mixing, then there is a factor map $\pi : (X,T) \to (F, S)$ onto a finite set $F$ with at least two elements on which $S : F \to F$ is a cyclic permutation.
\end{lemma}

\begin{proposition}
  \label{prop:chain-mixing-needed}
  Let $f\colon \AZ \to \AZ$ be a CA such that $f|_{\gls(f)} = \sigma|_{\gls(f)}$.
  Then $\gls(f)$ is a chain mixing subshift.
\end{proposition}

\begin{proof}
  By \cite[Prop. 5]{Ilkka}, $\gls(f)$ is chain transitive.
  If it is not chain mixing, we obtain a contradiction from \cref{lem:Akin} and \cref{coro:chaintrans}.
\end{proof}

To the previous obstructions on the language complexity of generic limit sets, we add the following.
Note the difference between shift-minimality (not properly containing a subshift) and inclusion-minimality (not properly containing an attractor of the CA): the former is a subshift-related property, while the latter is an attractor-related property.

\begin{proposition}\label{pi2b}
	Let $f$ be a CA. If $\gls(f)$ is inclusion-minimal, then its language is $\Pi_2^0$.
\end{proposition}

\begin{proof}
	The condition in \cref{it:E} of \cref{lemma4} is a $\Pi^0_2$ sentence:
	\[
	\forall s \in \lang(\gls(f)), v \in \A^*, i \in \Z, \forall T \in \N, \exists t > T, f^t([v]_{i}) \cap [s] \neq \emptyset.
	\]
	Since $f$ is computable, given all the parameters as input, checking whether $f^t([v]_{i}) \cap [s]$ is empty can be done with a Turing Machine. The full predicate is therefore $\Pi_2^0$.
\end{proof}

\begin{proposition}
  \label{prop:minimal-is-minimal}
  Let $f$ be a CA on $\A^\Z$.
  If $\gls(f)$ is shift-minimal, then it is inclusion-minimal.
\end{proposition}

\begin{proof}
  Suppose on the contrary that $X = \gls(f)$ is not inclusion-minimal.
  Then it properly contains a closed set $K \subset X$ with a nonmeager basin of attraction $B = \realm(K) \subset \A^\Z$ that is not comeager either (because $\gls(f)$ is the generic limit set).
  Since $K$ is closed, there exists $v \in \lang(X)$ with $K \cap [v] = \emptyset$.
  As $B$ has the Baire property, it is comeager in some nonempty open set, which we can choose to be a cylinder set $[w]_j \subset \A^\Z$ where $w$ is not the empty word.
  Our goal is to show that $v$ occurs periodically in every configuration of $X$, use these occurrences to construct a factor map onto a finite dynamical system, and obtain a contradiction with \cref{coro:chaintrans}.
  
  Denote $p = |w| > 0$.
  For each $n \geq 0$, the set $B_n = \bigcap_{i = -n}^n \sigma^{i p}(B)$ is comeager in the cylinder set $[w^{2n+1}]_{j-np}$ and is contained in the basin of $K_n = \bigcap_{i = -n}^n \sigma^{i p}(K)$.
  In particular, each $K_n$ is nonempty, hence their intersection $K' = \bigcap_{i \in \Z} \sigma^{i p}(K) \subset X$ is nonempty as well.
  Since $K$ is disjoint from $[v]$, we have $K' \subset \bigcap_{i \in \Z} (\A^\Z \setminus [v]_{i p})$.
  For $P \subset \Z$, define $X(P) = X \cap \bigcap_{i \in P} (\A^\Z \setminus [v]_i)$.
  We saw that $X(P) \neq \emptyset$ for some infinite subgroup $P = p\Z \subset \Z$. If we had $p=1$, $X(P)$ would be closed and stable by $\sigma$, hence a subshift contained in $X$. By shift-minimality of $X$, this means $X(P) = X$, which contradicts the nonemptiness of $[v] \cap X$. Thus $p \geq 2$.
  
  Let $q \geq 2$ be minimal such that $X(q \Z) \neq \emptyset$: there are configurations of $X$ with no subword $v$ starting on indices in $q\Z$.
  For $x \in X$, let $C(x) = \{ q \Z + i \mid x \in X(q \Z + i) \}$ be the set of cosets on which $x$ does not contain occurrences of $v$.
  Then $C(\sigma(x)) = C(x) + 1 = \{ q \Z + i + 1 \mid q \Z + i \in C(x)\}$ for all $x$.
  The number $|C(x)|$ is the same for all $x \in X$: the set $X'$ of those configurations $x$ for which $|C(x)|$ is maximal forms a subshift of $X$, and by shift-minimality of $X$ we have $X' = X$.
  Denote $m = |C(x)|$.
  
  The sets $C(x), C(\sigma(x)), \ldots, C(\sigma^{q-1}(x))$ are distinct for all $x \in X$: if $C(x) + a = C(\sigma^a(x)) = C(\sigma^b(x)) = C(x) + b$ for some $0 \leq a < b < q$, then $C(x) + (b-a) = C(x)$, meaning that $\sigma^{-a}(x) \in X((b-a)\Z)$, contradicting the minimality of $q$.
  Also, there exists $r \geq 0$ such that $C(x)$ only depends on $x_{[-r,r]}$: otherwise for all $r \geq 0$ there would exist $x(r) \in X \cap \bigcap_{i \in I} X(q [-r,r] + i)$ with $I \subset [0, q-1]$ of cardinality at least $m+1$, and a limit point $x$ of $(x(r))_{r \geq 0}$ would satisfy $C(x) \geq m+1$, a contradiction.
  All in all, we have shown that $C\colon (X, \sigma) \to (2^{\{0, \ldots, q-1\}}, {+1})$ is a morphism of dynamical systems whose image is finite and contains no fixed points.
  Since $X$ is shift-minimal, it is chain transitive; this contradicts \cref{coro:chaintrans}.
\end{proof}

Together with \cref{pi2b}, this result implies that a shift-minimal generic limit set has a $\Pi^0_2$ language.
In Proposition 1 of \cite{Ilkka} it was proved to be $\Sigma^0_2$, hence it must be $\Delta^0_2$.
Alternatively, one can apply the folklore result that every minimal $\Pi^0_k$ subshift is $\Delta^0_k$.

\begin{corollary}
  \label{cor:minimal-delta02}
  Let $f$ be a CA.
  If $\gls(f)$ is shift-minimal, then its language is $\Delta^0_2$.
\end{corollary}

The following result generalizes \cite[Prop. 2]{Ilkka}, where $f$ was required to be equicontinuous when restricted to $\gls(f)$.
To see that it really is a generalization, let $r$ be the radius of $f$ and recall that if the restriction $f|_{\gls(f)}$ is equicontinuous, then it is periodic with some period $p > 0$.
Choose a $\gls(f)$-forcing word $w \in \A^*$ with $f^t([w]_j) \subset [\lang_{r(2p+1)}(\gls(f))]$ for all $t \geq T$.
Then every $x \in [w]_{j-pr}$ satisfies $f^{t+p}(x)_{[0,r-1]} = f^t(x)_{[0,r-1]}$ for all $t \geq T$.
Thus we can extend $w$ into a blocking word and thus $f$ has an equicontinuity point on $\A^\Z$.

\begin{proposition}\label{s1b}
  Let $f$ be a CA with equicontinuity points. The language of its generic limit set $\gls(f)$ is $\Sigma_1^0$.
\end{proposition}

\begin{proof}
  Recall that a one-dimensional CA with equicontinuity points has at least one blocking word.
  Let $b \in \A^{2k+1}$ be a blocking word for $f$, and let $v_n \in \A^r$ for $n \in \N$ be the associated sequence of words, which is eventually periodic: there are $N \geq 0$ and $p > 0$ with $v_{n+p} = v_n$ for all $n \geq N$.
  Let $i \in \N$ and consider a configuration $x \in [b]_{-i-k} \cap [b]_{i-k}$.
  We have $f^n(x) \in = [v_n]_{-i} \cap [v_n]_i$ for all $n \in \N$.
  As $r$ is the radius of $f$, no information can flow over the $v_n$-words, so that the word $f^{n+1}(x)_{[-i+r, i-1]} \in \A^{2i-r+1}$ is completely determined by $f^n(y)_{[-i, i+r-1]}$ for all $n$.
  Since the sequence $(v_n)_{n \in \N}$ is $p$-periodic from index $N$ onward, the sequence $\alpha_i(x) := (f^n(x)_{[-i, i+r-1]})_{n \in \N}$ is $q$-periodic from index $N + p |\A|^{2i-r+1}$ onward for some $q \leq p |\A|^{2i-r+1}$, and it only depends on $x_{[-i-k, i+k-1]}$.

  Let $s \in \A^*$ be arbitrary.
  We claim that $s \in \lang(\gls(f))$ if and only if there exist $i \geq \max(|s|, r)$ and $N + p|\A|^{2i-r+1} \leq t \leq N + 2p|\A|^{2i-r+1}$ such that $f^{-t}([s]) \cap [b]_{-i-k} \cap [b]_{i-k} \neq \emptyset$.
  As this condition is $\Sigma^0_1$, the result follows.
  
  Suppose first that the latter condition holds for some $i$ and $t$, and let $x \in f^{-t}([s]) \cap [b]_{-i} \cap [b]_i$ be arbitrary.
  Denote $v = x_{[-i-k, i+k]}$, which begins and ends with $b$.
  We claim that $v$ enables $s$ in the sense of \cref{lemma2}.
  For this, pick any $u, w \in \A^*$, and let $y \in [u v w]_{-i-k-|u|}$ be arbitrary.
  Since $y \in [b]_{-i} \cap [b]_i$, the sequence $\alpha_i(y) = \alpha_i(x)$ is periodic from index $N + p |\A|^{2i-r+1}$ onward.
  Hence $f^n(y)_{[0, |s|-1]} = f^n(x)_{[0, |s|-1]} = s$ holds for infinitely many $n$, and $v$ enables $s$.

  Conversely, suppose that the latter condition does not hold: for all $i \geq \max(|s|, r)$ and $N + p|\A|^{2i-r+1} \leq t \leq N + 2p|\A|^{2i-r+1}$ we have $f^{-t}([s]) \cap [b]_{-i} \cap [b]_i = \emptyset$.
  We show that no word $v \in \A^*$ enables $s$.
  Pick any $j \in \Z$ and let $i \in \N$ be so large that there exists $x \in [b]_{-i-k} \cap [v]_j \cap [b]_{i-k}$.
  By assumption $f^t(x)_{[0, |s|-1]} \neq s$ for all $N + p|\A|^{2i-r+1} \leq t \leq N + 2p|\A|^{2i-r+1}$.
  The sequence $\alpha_i(x)$ is $q$-periodic with $q \leq p |\A|^{2i-r+1}$ from index $N + p |\A|^{2i-r+1}$, so $f^t(x)_{[0, |s|-1]} \neq s$ holds for all $t \geq N + p|\A|^{2i-r+1}$.
  Hence $v$ does not enable $s$.
\end{proof}

%
%

\section{Generic construction}
\label{sec:gener-constr}

In this section we present a construction of a CA $f$ which serves as a base for the CA built in \cref{sec:realizations} and \cref{sec:structure}, where within each proof modifications are introduced.
This type of construction first appeared in \cite{constructionfirst}; our presentation is based on \cite{boyer2015mu}.
An even more complex version was presented in \cite{construction}.

The main idea is the following:
the alphabet $\A$ of $f$ is the cartesian product of several auxiliary alphabets regarded as layers. The biinfinite tape, using these layers, is divided into individual finite computation zones called segments where the computations occur after the deletion of most of the initial data.
The computations depend on the application at hand: we simulate Turing machines in \cref{sec:realizations} and store patterns from subshifts in \cref{sec:structure}.
Depending on the construction, some segments may be merged with other segments as time passes.

The aforementioned layers of $\A$ are:
\begin{itemize}
\item
  \emph{Main Layer $\A_{\mathrm{main}}$.}
  Three special symbols are included: \emph{walls symbols} $\W$, \emph{initialization symbols} $\I$, and \emph{blank symbols} $\$$.
  An initialization $\I$-symbol is turned into a $\W$-symbol at the first step of the automaton, and two successive $\W$-symbols delimit areas of computation called segments.
  As time goes by, desired patterns are written on this layer as needed.
\item
  \emph{Computation Layer $\A_{\mathrm{comp}}$.}
  It encodes, in each segment as delimited on the Main Layer, a Turing Machine $\mathcal{M}$ which carries over the desired computations, and possibly other tasks.
  The simulated $\mathcal{M}$ writes the results of its computation on the Main Layer (the details vary depending on the application).
\item \emph{Cleaning Layer $\A_{\mathrm{clean}}$.}
  Using several types of signals, this layer erases any relic from the initial configuration.
\end{itemize}

On each alphabet we have a blank symbol which replaces data that is said to be `erased' -- for instance, in the Main Layer this role is played by $\$$. We denote by $\pi_{\mathrm{main}}$, $\pi_{\mathrm{comp}}$, and $\pi_{\mathrm{clean}}$ the projections on the Main, Computation, and Cleaning Layers, respectively. We also formally define the following:
\begin{definition}
  Let $x \in \A^\Z$ be a configuration and consider the forward orbit $(f^n(x))_{n \geq 0}$.
  A \emph{segment} in the initial configuration $x$ is a sequence of successive cells $s(i,j) = x_ix_{i+1}\ldots x_{j-1}x_j$ such that $\pi_{main}(x_i)=\pi_{main}(x_j) = \I$ and $\pi_{main}(x_k) \neq \I$ for all $i < k < j$.
  For $n \geq 1$, a segment in $y = f^n(x)$ is a sequence $y_iy_{i+1}\ldots y_{j-1}y_j$ such that $y_i = y_j = \W$ and $s(i,j)$ is a segment of $x$.
\end{definition}

In order for all segments to start and perform their computations without disruption, it is necessary to clean the data on all layers in the initial configuration, with the exception of $\I$-symbols in $\A_{\mathrm{main}}$ -- which initiate the cleaning and the segments' internal processes, and are immediately turned into $\W$-symbols at the first step of the CA. Observe that walls may also be present in the initial configuration, \ie, some $\W$-symbols are not created by an $\I$-symbol. These walls need to be deleted.

The deletion process is carried out by signals $s_i$ and $s_o$ (inner and outer) generated by every initialization symbol $\I$ in both directions. They are erased once they meet their counterpart coming form another $\I$-symbol, and they delete any walls and other data they encounter. These signals are defined similarly to \cite[Section 3]{boyer2015mu}: the outer signal $s_o$ travels faster than $s_i$; $s_o$ deletes everything it encounters on each layer that is not another $s_o$; when two $s_o$ signals collide, they send auxiliary signals that bounce back on the inner signals $s_i$ behind them and return to the collision point. If the bouncing signals do not meet again at the same time step, the $(s_o,s_i)$ pair from which the latter one came has a greater gap between its two signals, meaning this pair has not been generated at time $1$ -- this holds since signals $s_o$ and $s_i$ can not be both present in the same cell. The pair $(s_o,s_i)$ with the greater gap is consequently deleted by other auxiliary signals generated by the latter bouncing signal.

\begin{figure}[H]
	\centering
	\begin{subfigure}[t]{0.48\textwidth}
		\centering
		\resizebox{\columnwidth}{!}{
			\begin{tikzpicture}
				\draw[black] (0,1) -- (8,1);
				\draw[black, line width=1.5] (8,1) -- (8,5);
				\draw[black, line width=1.5] (0,1) -- (0,5);
				\draw[black, line width=1.5] (3,1) -- (3,1.75);
				
				\draw[black, thick] (0,1) -- (4,2) -- (8,1);
				\draw[black!25, thick] (0,1) -- (4,5) -- (8,1);
				
				\draw[black, thick] (6,2.5) -- (4,2) -- (2,2.5);
				\draw [black, thick, dotted] (6,2.5) -- (6.4,2.6);
				\draw [black, thick, dotted] (2,2.5) -- (1.6,2.6);
				
				\draw[black!50] (4,2) -- (1.23,2.2) -- (4,2.4) -- (6.77,2.2) -- (4,2);
				\draw[black!50, thick, dashed] (4,2) -- (4,2.4);
				
				\draw[black, line width=0.25] (6,2.5) -- (4,2.4) -- (2,2.5);
				
				\draw[black!50] (7.1,1.22) -- (4.2,1);

				\draw (0,0.8) node {$\mathcal{I}$};
				\draw (8,0.8) node {$\mathcal{I}$};
				\draw (3,0.8) node {$\mathcal{W}$};
				\draw (2,1.3) node {$s_o$};
				\draw (5,1.6) node {$s_o$};
				\draw (0.7,2) node {$s_i$};
				\draw (7.3,2) node {$s_i$};
		\end{tikzpicture}}
		\caption{Pairs of signals $s_o$ and $s_i$ from two adjacent $\mathcal{I}$'s meet. The $s_o$'s erase everything else of the original configuration in the segment (here, a starting $\W$ and a lonely auxiliary signal).}\label{Comparison1}
	\end{subfigure}
	\hfill
	\begin{subfigure}[t]{0.48\textwidth}
		\centering
		\resizebox{\columnwidth}{!}{
			\begin{tikzpicture}
				\draw[black] (0,1) -- (8.5,1);
				\draw[black, line width=1.5] (0,1) -- (0,5);

				\draw[black, thick] (0,1) -- (4,2) -- (8,1);
				\draw[black!25, thick] (0,1) -- (4,5);
				\draw[black!25, thick] (8.5,1) -- (6.865,2.75);
				
				\draw[black, thick] (6,2.5) -- (4,2) -- (2,2.5);
				\draw [black, thick] (6,2.5) -- (7.8,3);
				\draw [black, thick, dotted] (2,2.5) -- (1.6,2.6);
				
				\draw[black!50] (4,2.4) -- (1.23,2.2) -- (4,2);
				\draw[black!50] (4,2) -- (7.18,2.355) -- (4,2.705);
				
				\draw[black, thick, dashed] (4,2) -- (4,2.4);
				\draw[black, thick, dotted] (4,2) -- (4,2.705);
				
				\draw[black, line width=0.25] (4,2.4) -- (2,2.5);

				\draw (0,0.8) node {$\mathcal{I}$};

				\draw (7.4,3.1) node {$s_o$};
				\draw (6,1.3) node {$s_o$};
				\draw (3.2,4.6) node {$s_i$};
				\draw (7.9,2) node {$s_i$};
		\end{tikzpicture}}
		\caption{``Wild'' signals from the original configuration cannot disrupt a pair of $s_o$ and $s_i$ coming from an $\mathcal{I}$. The slope of some auxiliary signals is slightly exaggerated for the phenomenon of them bouncing back not at the same time to be more visible.}\label{Comparison2}
	\end{subfigure}
	\caption{Space-time diagram of the deleting process.}
	\label{Comparison}
\end{figure}
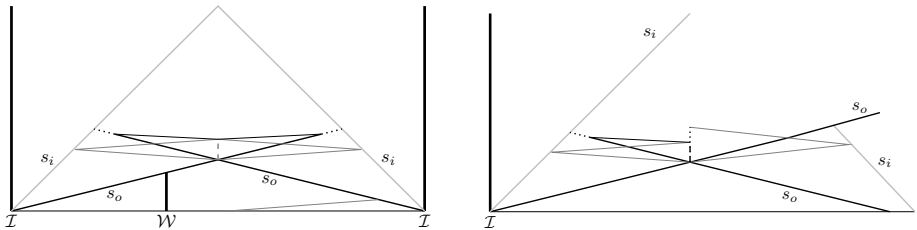

Just as the construction in \cite{boyer2015mu} protects specific $*$-states, our construction protects walls originating from $\I$-states, and deletes any other wall. The notable distinction with \cite{boyer2015mu} is that here, when signals from two $\mathcal{I}$-symbols collide, they merely vanish.
Note that these signals $s_o$ and $s_i$ need to move slower than speed $1$ (one cell at each time step) for the process with bouncing signals to go smoothly: speeds $1/4$ and $1/5$ work according to \cite{boyer2015mu}. This requires several states for the signals.

At time step 1, as it is turned into a $\W$-symbol and launches signals $s_i$ and $s_o$, each $\I$-symbol also starts an internal computation process on $\A_{\mathrm{comp}}$ in its associated segment. These internal computation processes vary for each construction, but in any case they have a clean canvas to perform any computation needed, as the outer signals $s_o$ will replace the contents of each correctly initialized segment with blank symbols.

%
%

\section{Realization of complexity for equicontinuity points}
\label{sec:realizations}

In this section, we realize a CA $f$ which realizes the bound in \cref{s1b}, that is, $f$ has equicontinuity points and the language of $\gls{(f)}$ is $\Sigma^0_1$-complete. Moreover, we show that such an $f$ can be built so that it acts as the identity on its generic limit set.
By the remark before \cref{s1b}, such an $f$ necessarily has equicontinuity points.

\begin{theorem}\label{equiident}
There exists a CA $f$ such that $\lang(\gls(f))$ is a $\Sigma^0_1$-complete set, and $f|_{\gls(f)} = \mathrm{id}|_{\gls(f)}$.
\end{theorem}

\begin{proof}
We describe a CA $f$ with the desired properties. Consider the construction from \cref{sec:gener-constr}, modified as follows.
\begin{itemize}
\item The Main Layer's alphabet is $\{0,1,\$,\I,\W\}$, where $\$$ is the blank symbol.
\item The only way to erase a wall is with an outer signal $s_o$. In particular, walls created by an $\I$ always remain, so that segments formed between two of them stay forever.
\item In addition to the signals $s_o$ and $s_i$, each $\I$ initializes a simulated computation of a Turing Machine $\M$ on the segment to its left.
\item The Cleaning Layer and its deleting process, described in \cref{sec:gener-constr}, remain untouched.
\item As the cleaning and deleting processes take place, all the information in any segment is replaced by $\$$-symbols.
\end{itemize}

As for the behavior of the machine $\M$, consider an enumeration of all Turing Machines $(\M_n)_{n\in\N}$ with a one-way infinite tape, and consider a computable bijection $p:\N\to\N\times\N$ -- for instance, the inverse of the Cantor pairing function, but we modify it so that we avoid any case where $\ell-(|\mathrm{bin}(n)|+1) <0$ with $(n,m)=p(\ell)$.

In each segment, $\M$ starts by determining the length $\ell$ of its segment (by sending a specific signal and waiting for its return, for instance) and computes $(n,m)=p(\ell)$. Then $\M$ simulates $m$ steps of computation of the machine $\M_n$ on the empty input. If $\M_n$ halts during these $m$ simulated steps, then $\M$ prints $\mathrm{bin}(n) \in \{0,1\}^*$, the binary representation of $n$, on the left end of the segment, leaving one blank cell between it and the left wall, and fills the rest of the segment with blank symbols. If $\M_n$ does not halt in at most $m$ steps of computation, $\M$ fills the segment with blank symbols. In both cases, once the described computations are done, $\M$ is deleted. In this manner, every segment is eventually of the form $\W \$ \mathrm{bin}(n) \$^k \W$ or $\W \$^\ell\W$, with $k=\ell-(|\mathrm{bin}(n)|+1)$ (notice that $p$ is designed so that $k \geq 0$).
The segment remains unchanged from that point on.

We first claim that $\gls(f) \subset {\left(\{\W, 0, 1, \$\}\times\{\$\}\times\{\$\}\right)}^\Z$.
Once proved, this implies ${f|_{\gls(f)} = \mathrm{id}|_{\gls(f)}}$, since $f$ acts as the identity on the above full shift.
Let $s \in \lang(\gls(f))$ be enabled by some cylinder set $[v]_i$ as per \cref{lemma2}.
We may assume, by extending $v$ if necessary, that $i \leq 0$ and $|v| \geq |i| + |s|$.
Choose $u = w = \I$.
Then for any configuration $x \in [u v w]_{i-1}$, the Cleaning Layer ensures that the word $f^t(x)_{[i-1,i+|v|+1]} \in \W \A^{|v|} \W$ consists of correctly initialized segments for all large enough $t \geq 1$.
The instances of the machine $\M$ simulated on the Computation Layer will eventually fill each segment with symbols from $\{0,1,\$\}$ and disappear.
Thus $f^t([u v w]_{i-1}) \subset [\W \{\W,0,1,\$\}^{|v|} \W]_{i-1}$ for all large enough $t$, and infinitely many of them contain $[s]$ due to $v$ enabling $s$, and thus $s \in \{\W,0,1,\$\}^*$.

Let $n \in \N$.
We claim that $s_n =\$ \mathrm{bin}(n) \$ \in \lang(\gls(f))$ if and only if $\M_n$ eventually halts.
First, if $\M_n$ never halts, then no correctly initialized segment will contain the word $s_n$.
By the analysis in the previous paragraph and the construction above, $s_n \notin \lang(\gls(f))$.
Suppose now that $\M_n$ halts in some $m$ steps.
Since $p$ is a bijection from $\N$ to $\N\times\N$, there exists $\ell$ such that $(n,m)=p(\ell)$.
We show that $s_n$ is enabled by the cylinder $C = [\I \$^\ell \I]_{-1}$.
For all $x \in C$ and $t \geq 1$, the word $f^t(x)_{[-1, \ell]} \in \W \A^\ell \W$ is a correctly initialized segment; and the instance of $\M$ it contains simulates $m$ steps of $\M_n$ on that segment. When $\M_n$ halts, $\M$ writes $\$ \mathrm{bin}(n) \$^{\ell-(|\mathrm{bin}(n)|+1)}$ on the segment and disappears.
Hence $f^t(x) \in [\$ \mathrm{bin}(n) \$]$ for all large enough $t$.

Since the set of Turing Machines that eventually halt on the empty input is known to be $\Sigma_0^1$-complete, we have built the expected CA.
\end{proof}

\section{Realization of structure}
\label{sec:structure}

\subsection{Statement and auxiliary results}

In this section, we realize two large classes of $\Pi^0_2$ subshifts as generic limit sets of cellular automata.
More specifically, we prove the following result:

\begin{theorem}
  \label{thm:MixingRealization}
  Let $X \subset \A^\Z$ be a chain mixing subshift satisfying one of the following conditions:
  \begin{enumerate}
  \item
    either $X$ is $\Pi^0_2$ and contains a nonempty $\Pi^0_1$ subshift;
  \item
    or $X$ is $\Delta^0_2$.
  \end{enumerate}
  Then there exists an alphabet $\B \supset \A$ a CA $f : \B^\Z \to \B^\Z$ with $\gls(f) = X$ and $f|_X = \sigma|_X$.
\end{theorem}

We prove the two cases of the theorem simultaneously, pointing out the (relatively minor) differences in the construction and proofs whenever they diverge; we call them the $\Pi^0_1 \subset \Pi^0_2$ case and the $\Delta^0_2$ case.
Though in a more complex fashion than \cref{sec:realizations}, the construction is also based on the walls-and-counters CA of \cref{sec:gener-constr}.
The role of each segment is again to help ensure that some specific word occurs in the generic limit set.
As small segments merge with larger ones in a process described in \cref{subsec:merge}, the words generated by the former are overwritten by those of the latter.
The main challenge is to implement this rewriting in such a way that it does not produce additional words in the generic limit set.
The chain mixing property, the existence of a nonempty $\Pi^0_1$ subshift, and the $\Delta^0_2$ complexity help us ensure this.
The chain mixing property is in fact necessary here, by \cref{prop:chain-mixing-needed}, since the construction below obeys its hypothesis.

We begin with two technical lemmas on the structure of the kinds of subshifts that appear in the statement of \cref{thm:MixingRealization}.

\begin{lemma}
  \label{lem:Pi02SFTApprox}
  Let $X \subset \A^\Z$ be a nonempty chain mixing $\Pi^0_2$ subshift, and $Y \subset X$ a nonempty $\Pi^0_1$ subshift.
  Then there exists a sequence $(X_m, Y_m, w_m)_{m \in \N}$, where each $X_m \subset \A^\Z$ is a mixing SFT, $Y_m \subset X_m$ is a nonempty SFT, and $w_m \in \lang(X_m)$, such that the following conditions hold.
  \begin{enumerate}
  \item
    \label{it:desc}
    $Y = \bigcap_{m \in \N} Y_m$ and $Y_{m+1} \subset Y_m$ for all $m \in \N$.
  \item
    \label{it:capcup}
    $\lang(X) = \bigcap_{M \in \N} \bigcup_{m \geq M} \lang(X_m) = \bigcap_{M \in \N} \bigcup_{m \geq M} \{w_m\}$.
  \item
    \label{it:window}
    For each $m$, the window size and mixing distance of $X_m$, the window size of $Y_m$, and the length $|w_m|$ are all $o(\log m)$.
  \item
    \label{it:space}
    The function $m \mapsto (X_m, Y_m, w_m)$ is computable in $O(2^m)$ space.
  \end{enumerate}
\end{lemma}

\begin{proof}
  Since $X$ is $\Pi^0_2$, there exists a computable predicate $\phi_X$ such that $\lang(X) = \{ w \in \A^* \mid \forall k\ \exists \ell \  \phi_X(w, k, \ell) \}$.
  Since $Y$ is $\Pi^0_1$, there exists a computable predicate $\phi_Y$ such that $\lang(Y) = \{ w \in \A^* \mid \forall k \  \phi_Y(w, k) \}$.
  We first describe an algorithm that produces a sequence of SFTs satisfying the first two items.
  Then we modify it to satisfy the remaining items as well.

  The algorithm keeps track of three finite sets of words $M, Q, F \subset \A^*$, which we call the memory, the queue, and the forbidden set.
  All three sets are initially empty.
  The memory and queue are used to construct the $X_m$ and $w_m$, while the forbidden set is used for $Y_m$.
  For each $w \in M$, the algorithm also stores numbers $k_w, \ell_w \in \N_0$, and for each $w \in Q$ it stores a number $k'_w \in \N_0$.

  The algorithm proceeds in rounds, starting from $i = 0$.
  Round $i$ consists of the following steps:
  \begin{enumerate}
  \item
  \label{st:fillM}
    Add a new word $u \in \A^* \setminus M$ to $M$; they are added in increasing order of length, and lexicographically for a given length.
    Set $k_u = k'_u = \ell_u = 0$.
  \item
  \label{st:firing}
    For each $w \in M$, check whether $\phi_X(w, k_w, \ell_w)$ holds.
    If it does, we say that $w$ fires, and we increment $k_w$ and set $\ell_w = 0$.
    If $w$ is not an element of $Q$, we also add it to $Q$ and set $k'_w = k_w$.
    If $\phi(w, k_w, \ell_w)$ does not hold, we increment $\ell_w$.
  \item
  \label{st:output}
    For each $w \in Q$, do the following.
    Let $Y'$ be the SFT defined by forbidding all words in the forbidden set $F$.
    Denote $p = \max(k'_w, |w|)$ and $F_p = \{ v \in \A^{\leq p} \mid k_v \leq p \}$.
    If the SFT defined by forbidding $F_p$ contains a mixing sub-SFT $X'$ with $w \in \lang(X')$ and $Y' \subset X'$, then remove $w$ from $Q$ and output the triple $(X', Y', w)$.
  \item
    \label{st:buildY}
    For each $w \in \bigcup_{j \leq i} \A^j$, if there exists $k \leq i$ such that $\phi_Y(w, k)$ does not hold, then add $w$ to $F$.
  \end{enumerate}
  The algorithm executes these rounds in an infinite loop.
  It outputs a sequence of triples, which we denote by $(X_m, Y_m, w_m)_{m \in \N}$.

  By construction, each $X_m$ produced by the algorithm is a mixing SFT with $w_m \in \lang(X_m)$ and $Y_m \subset X_m$.
  Since the algorithm never removes words from $F$, the sequence $(Y_m)_{m \in \N}$ is decreasing.
  Step~\ref{st:buildY} of each round guarantees that every $w \in \A^* \setminus \lang(Y)$ is eventually added to $F$, so \cref{it:desc} of the statement holds.
  
  Consider then a word $w \in \lang(X)$.
  Due to the definition of $\phi_X$, it fires an infinite number of times during the execution of the algorithm.
  Whenever $w$ fires and is not in the queue, it is added there and the number $k'_w$ is fixed for all rounds until $w$ leaves the queue.
  Consider then $p = \max(|w|, k'_w)$ and the set $F_p$; they are both fixed until $w$ leaves the queue.
  We prove that $w$ does leave the queue after some round.
  
  Because $X$ is chain mixing, the SFT approximation $\mathcal{S}_p(X)$ is mixing, and its language contains $w$ since $p \geq |w|$.
  If $w$ did not leave the queue before that due to some output $(X_m, Y_m, w_m)$ with $w_m = w$, each word $u$ in $\lang_{\leq p}(X)$ will eventually leave $F_p$ (because its $k_u$ grows to infinity with the rounds), and each word in $\A^{\leq p} \setminus \lang(Y)$ will eventually enter $F$ (by definition of Step~\ref{st:buildY}).
  Once this happens, we have $Y' \subset \mathcal{S}_p(X)$.
  Thus we can choose $\mathcal{S}_p(X)$ as $X'$ if a suitable mixing SFT was not found earlier, and the algorithm outputs $(X', Y', w)$ as $(X_m, Y_m, w_m)$ for some $m$.
  Therefore $w$ is removed from the queue after some round.
  
  Furthermore, any such $w$ is added again at a later round when it eventually fires anew, since $w \in \lang(X)$.
  Thus $w \in \bigcap_{M \in \N} \bigcup_{m \geq M} \{w_m\}$.
  In particular the algorithm produces an infinite sequence of triples.
  
  Now, take a word $w \notin \lang(X)$, which fires only a finite number of times.
  Denote $n = |w|$.
  After some number of rounds, each word $u \in \A^{\leq n} \setminus \lang(X)$ has fired for the last time and the value $k_u = k'_u$ has settled into a constant. These words may leave the queue once more, but produce a finite number of outputs (to which $w$ may belong) by doing so.
  
  Let $K_n = \max \{ k_u \mid u \in \A^{\leq n} \setminus \lang(X) \}$.
  After a bigger number of rounds, whenever a new word $v \in \A^*$ fires and enters the queue, we have either $|v| > K_n$ or $k'_v > K_n$. Indeed, after some point, we have that $k_v > K_n$ holds for all $v \in \lang_{\leq K_n}(X)$.
  This means that $w$ will never leave the set $F_p$ for $p = \max(|v|, k'_v)$, and thus does not occur in the mixing SFT $X'$ if one is produced for such a $v$.
  
  Hence $w$ belongs to a finite number of $\lang(X_m)$, and as such we conclude that $w \notin \bigcap_{M \in \N} \bigcup_{m \geq M} \lang(X_m)$.
  We have shown that \cref{it:capcup} is satisfied, considering its remaining inclusions are obvious.

  Next, we modify the algorithm so that it produces a modified sequence $(X_{s(m)}, Y_{s(m)}, w_{s(m)})_{m \in \N}$,~where $s\colon \N \to \N$ is a nondecreasing computable function with $s(m) \leq s(m+1) \leq s(m)+1$ for all $m$.
  All such sequences satisfy the first two conditions.
  Since the mixing distance and window size of $X_m$, the window size of $Y_m$, the length $|w_m|$ and the space used by the unmodified algorithm are all computable from $m$, we can choose $s$ to grow slowly enough so that the remaining conditions, \cref{it:window} and \cref{it:space}, hold as well.
\end{proof}

\begin{lemma}
  \label{lem:Delta02SFTApprox}
  Let $X \subset \A^\Z$ be a nonempty chain mixing $\Delta^0_2$ subshift.
  Then there exists a sequence $(X_m, w_m)_{m \in \N}$, where each $X_m \subset \A^\Z$ is a mixing SFT and $w_m \in \lang(X_m)$, such that the following conditions hold.
  \begin{enumerate}
  \item
    \label{it:capcupcap2}
    $\lang(X) = \lim_{m \in \N} \lang(X_m) = \bigcap_{M \in \N} \bigcup_{m \geq M} \{w_m\}$.
  \item
    \label{it:window2}
    For each $m$, the window size and mixing distance of $X_m$ and the length $|w_m|$ are all $o(\log m)$.
  \item
    \label{it:space2}
    The function $m \mapsto (X_m, w_m)$ is computable in $O(2^m)$ space.
  \end{enumerate}
\end{lemma}

\begin{proof}
	Since $X$ is $\Delta^0_2$, there are two computable predicates $\phi_X^+$ and $\phi_X^-$ such that $\lang(X) = \{ w \in \A^* \mid \forall k\ \exists \ell \  \phi_X^+(w, k, \ell) \} = \{ w \in \A^* \mid \exists k\ \forall \ell \ \neg \phi_X^-(w, k, \ell) \}$.
  Consider the predicate $\phi_X(w, n)$ defined as follows.
  \begin{enumerate}
  \item
    Starting from $k = 0$, check for increasing $\ell \geq 0$ whether $\phi_X^+(w, k, \ell)$ holds, and whenever it does, increment $k$ and reset $\ell$ to $0$.
    Do this until $n$ pairs $(k, \ell)$ have been checked, and let $k_+$ be the final value of $k$.
  \item
    Do the same for $\phi_X^-$ in place of $\phi_X^+$, and let $k_-$ be the final value of $k$.
  \item
    Define $\phi_X(w, n)$ as the truth value of $k_+ > k_-$.
  \end{enumerate}
  Then for any $w \in \lang(X)$, $\phi_X(w, n)$ holds for all large enough $n$, while for $w \in \A^* \setminus \lang(X)$, $\neg \phi_X(w, n)$ holds for all large enough $n$.

  We describe an algorithm that is very similar to that of \cref{lem:Pi02SFTApprox}.
  It stores a finite memory $M \subset \A^*$, which is initially empty.
  It proceeds in rounds, with round $i$ consisting of the following steps.
  \begin{enumerate}
  \item
    Add a new word $w \in \A^*$ into $M$, in increasing order of length.
  \item
    Let $Q = \{ w \in M \mid \phi_X(w, i) \}$, $F = M \setminus Q$ and $n = \max \{|w| \mid w \in M\}$.
    For each $w \in Q$, do the following.
    If there exists $|w| \leq p \leq n$ such that the SFT $X_p$ defined by forbidding the words $F \cap \A^{\leq p}$ is mixing and satisfies $\lang_j(X_p) = Q \cap \A^j$ for each $j \leq p$ and $w \in \lang(X_p)$, choose the largest such $p$ and output $(X_p, w)$.
  \end{enumerate}

  We claim that the sequence $(X_m, w_m)_{m \in \N}$ produced by the algorithm satisfies \cref{it:capcupcap2}; the others follow as in \cref{lem:Pi02SFTApprox}.
  Given $k \geq 0$, let $i_0 \geq |\A|^k$ be so large that for all $v \in \A^{\leq k}$ and $i \geq i_0$, $\phi_X(v, i)$ holds if and only if $v \in \lang(X)$.
  Such an $i_0$ exists since $\phi_X(v, i)$ converges to the correct value for each $v \in \A^{\leq k}$ separately and $\A^{\leq k}$ is a finite set.
  Then the SFT forbidding $F \cap \A^{\leq k}$ is precisely the SFT approximation $\mathcal{S}_k(X)$, which is mixing by assumption.
  
  Suppose $i \geq i_0$ and consider an output $(X_p, w)$ produced on step 2 of the algorithm on round $i$.
  We have $n = \max \{|w| \mid w \in M\} \geq k$.
  If $|w| \leq k$, then $w \in Q$ implies $w \in \lang(X)$, and in this case $p \geq k$, since $|w| \leq p \leq n$, $p$ is chosen as large as possible, and $k$ is a valid choice.
  If $|w| > k$, then we have $p \geq k$ by definition.
  In either case, for each $j \leq k$ we have $\lang_j(X_p) = Q \cap \A^j$ by definition of $X_p$; which is equal to $\lang_j(\mathcal{S}_k(X)) = \lang_j(X)$ since $i \geq i_0$.
  Thus we have shown $\lang_j(X_p) = \lang_j(X)$ for all $0 \leq j \leq k$ and all pairs $(X_p, w)$ produced after round $i_0$.
  This implies $\lang(X) = \lim_{m \in \N} \lang(X_m)$.

  Consider then $k \geq 0$ and a word $w \in \lang_k(X)$.
  If $i \geq i_0$, then on step 2 of round $i$ of the algorithm, $w \in Q$ and $p = |w|$ is a valid choice for $w$.
  Hence $w = w_m$ for infinitely many $m$'s.
  On the other hand, for each $w \in \A^k \setminus \lang(X)$ we have $w \notin Q$ for all $i \geq i_0$.
  Hence $w = w_m$ for only finitely many $m$'s.
  This proves $\lang(X) = \bigcap_{M \in \N} \bigcup_{m \geq M} \{w_m\}$.
\end{proof}

For the next lemma, we recall some terminology from combinatorics on words.
A set $C \subset \A^*$ is a \emph{code}, if $c_1 \cdots c_m = c'_1 \cdots c'_n$ with $c_i, c'_i \in C$ implies $m = n$ and $c_i = c'_i$ for all $0 \leq i < m$.
A word $w \in \A^*$ is \emph{primitive} if $w = z^n$ implies $n = 1$.
The \emph{conjugates} of $w \in \A^n$ are the words $w_{[i, n-1]} w_{[0, i-1]}$ for $0 \leq i < n$, and $w$ is a \emph{Lyndon word} if it is primitive and lexicographically minimal among its conjugates.
Finally, $w$ is \emph{unbordered} if no prefix of $w$ is a suffix of $w$.

\begin{lemma}
  \label{lem:LeastPeriod}
  Let $X \subseteq \A^\Z$ be an infinite mixing SFT with window size and mixing distance $k$, and let $W \subset \lang(X)$ be finite.
  Denote $N = k(|W|-1) + \sum_{w \in W} |w|$.
  For any $n > 2N + 8k$, there exists a periodic configuration $x \in X$ with least period $n$ such that $w \in \lang(x)$ for all $w \in W$.
\end{lemma}

\begin{proof}
  We first prove that for each $m \geq 2k$, there exists an unbordered word $v \in \lang(X)$ with $m \leq |v| < m+2k$.
  Consider the width-$k$ Rauzy graph $G$ of $X$ with edge labels in $A$.
  Pick any vertex $p \in G$ and consider the set $C \subset \lang(X)$ of first returns from $p$ to itself, which is a code.
  Since $k$ is a mixing distance for $X$, there exists $c \in C$ with $|c| \leq k$.
  Since $X$ is infinite, there exists another first return $c' \in C$, which is either shorter than $c$, or satisfies $c'_i \neq c_i$ for some $0 \leq i < k$.
  In the first case we set $d = c'$, and in the latter we extend the prefix $c'_0 \cdots c'_i$ into a first return $d \in C$ with $|d| < 2 k$.
  As $C$ is a code, $c$ and $d$ are not powers of the same word.
  Then $c^\ell d^\ell \in \lang(X)$ is primitive for all $\ell \geq 2$ \cite[Theorem 9.2.4]{Lo97}, so one of its conjugates $v \in \lang(X)$ is a Lyndon word, hence unbordered by \cite[Proposition 5.1.2]{Lo97}.
  The claim on $|v|$ holds for $\ell = \lfloor m/|cd| \rfloor$.

  Denote $W = \{w_1, \ldots, w_{|W|}\}$.
  Since $k$ is a mixing distance for $X$, there exist gluing words $u_1, \ldots, u_{|W|-1} \in \lang_k(X)$ with $u = w_1 u_1 w_2 u_2 \cdots u_{|W|-1} w_{|W|} \in \lang_N(X)$.
  Let $n > 2N+8k$. Let $v \in \lang(X)$ be an unbordered word with $n-N-4k < |v| \leq n-N-2k$.
  As $k$ is also a window size for $X$, there exist gluing words $a \in \lang_k(X), b \in \lang_{n-|u a v|}(X)$ with $x = {}^\infty (u a v b)^\infty \in X$.
  Each $w \in W$ occurs in this configuration, since they occur in $u$.
  The least period of $x$ is $|u a v b| = n$, since $v$ is unbordered and $|v| > n/2$.
\end{proof}

The construction of the unbordered word $v$ in the above proof is essentially \cite[Lemma 2]{BePe09}.
We repeat it here, since we need finer control on the lengths of the words.

\begin{lemma}
  \label{lem:PeriodicSubwords}
  Let $n \geq 1$, and suppose that every word of length $2 n$ occurring in $x \in \A^\Z$ is $q$-periodic for some $1 \leq q \leq n$ (which might depend on the word).
  Then $x$ is $q$-periodic for some $1 \leq q \leq n$.
\end{lemma}

\begin{proof}
  Let $i \in \Z$ and $k \geq 2n-1$.
  We prove by induction on $k$ that $u = x_{[i, i+k]}$ is $q$-periodic for some $1 \leq q \leq n$.
  The claim follows when we let $k$ grow and choose $i = -\lfloor k/2 \rfloor$.
  
  The case $k = 2n-1$ is true by assumption, so suppose $k \geq 2n$.
  Denote $u = v w a$, where $v \in \A^+$, $w \in \A^{2n-1}$ and $a \in \A$.
  Then $v w$ is $p$-periodic and $w a$ is $q$-periodic for some $p, q \leq n$.
  Recall the periodicity theorem of Fine and Wilf~\cite{FiWi65}: if a word has periods $p$ and $q$, and length at least $p+q-\gcd(p,q)$, then it also has $\gcd(p,q)$ as a period.
  The word $w$ satisfies the conditions, because $p+q-\gcd(p,q)$ is at worst $\max(p,q)$, so $w$ is $\gcd(p,q)$-periodic.
  Then $v w$ and $w a$, and hence $u$, are also $\gcd(p,q)$-periodic, which is what we claimed.
\end{proof}

\begin{lemma}
  \label{lem:GluingTimeComplexity}
  Fix a finite alphabet $\A$.
  Given the Rauzy graph of a mixing SFT $X \subset \A^\Z$ with window size and mixing distance $m$, and two words $u, w \in \lang_m(X)$, the time complexity of computing a gluing word $v \in \A^m$ such that $u v w \in \lang(X)$ is at most $\exp(O(m))$.
\end{lemma}

\begin{proof}
  The nodes of the Rauzy graph $G$ of $X$ are words of length $m$, so its size is at most $|\A|^m$.
  Computing $v$ amounts to finding a length-$2m$ path from $u$ to $v$ in $G$.
  We perform a breadth-first search, computing for each $i = 0, 1, \ldots, m$ the set of vertices $C_i \subset G$ that are reachable from $u$ in exactly $i$ steps, and the set $D_i$ of vertices from which $w$ is reachable in exactly $i$ steps.
  Since the in- and outdegree of each vertex of $G$ is at most $|\A|$, we have $|C_i|, |D_i| \leq |\A|^i$, and $C_i$ and $D_i$ can be computed in time $\poly(|\A|^m \cdot |\A|^i) = |\A|^{O(m)}$.
  We can choose any word in $C_m \cap D_m$ as $v$, and finding one takes another $|\A|^{O(m)}$ steps.
\end{proof}

\subsection{Walls, counters and conveyor belts}

The high-level structure of the CA $f$ is the same for both cases of \cref{thm:MixingRealization}.
We define the alphabet $\B$ of the CA $f$ as a set larger than $\A$, which is the alphabet of the subshift we want to realize. The alphabet $\B$ consists of three layers as listed in the construction of \cref{sec:gener-constr} (using the letter $\B$ in place of $\A$): the Main Layer $\B_{\mathrm{main}}$, the Computation Layer $\B_{\mathrm{comp}}$, and the Cleaning Layer $\B_{\mathrm{clean}}$.
To define it, let $\M$ be a Turing machine with state set $Q$, initial state $q_0 \in Q$, tape alphabet $\Gamma$, and blank tape symbol $\gamma_0 \in \Gamma$.
We will describe the behavior of $\M$ later on; for now, we only need to name its components in order to define the alphabet of $f$.

The Main Layer has alphabet $\B_{\mathrm{main}} = \{\I, \$, \W_\$\} \cup \{\W_a \mid a \in \A\} \cup \A$.
By default, symbols of the subset $\A_1 := \A \cup \{\$\}$ are continually shifted to the left.
The ``decorated'' wall symbols $\W_a$ for $a \in \A_1$ behave exactly like the $\W$-symbols of \cref{sec:gener-constr}, and the decorations allow us to shift the symbols of $\A_1$ through the walls.
This allows a segment to receive data from another segment on its right in order to determine whether they should merge.
We identify with $\A$ the states $b \in \B$ such that $\pi_{\mathrm{main}}(b) \in \A$ and whose other layers are blank.
They will be the only states visible in the generic limit set, allowing for the realization of the desired subshift $X \subset \A^\Z$.

The Computational Layer of $f$ consists of four sub-layers, called the Right Conveyor Belt Layer, Comparison Layer, Turing Machine Layer, and Clock Layer.
It also contains a blank symbol, which we denote by $\#$.
The layers are denoted $\B_{\mathrm{comp}} = \B_{\mathrm{belt}} \times \B_{\mathrm{cmpr}} \times \B_{\mathrm{TM}} \times \B_{\mathrm{clock}} \cup \{\#\}$.
The projection maps from the components of $\B_{\mathrm{comp}}$ are undefined on $\#$.
The sub-layers are defined as follows.
\begin{itemize}
\item
  The Right Conveyor Belt Layer $\B_{\mathrm{belt}} = \A_1$ contains symbols from $\A$ and blank symbols.
  By default, it is continually shifted to the right.
  Together with the Main Layer, it forms ``conveyor belts'' on which circular words over $\A_1$ are cyclically shifted.
\item
  The Comparison Layer $\B_{\mathrm{cmpr}} = \A_1$ also contains symbols from $\A$ and blanks.
  By default, it is continually shifted to the left.
\item
  The Turing Machine Layer $\B_{\mathrm{TM}} = Q \cup \Gamma$ is used to simulate the machine $\M$.
\item
  We use the Clock Layer $\B_{\mathrm{clock}} = \{0,1,2,3\}$ to implement a ternary counter that times certain actions of $\M$.
\end{itemize}

We define the CA $f$ over the course of the next few sections.
We begin by stating ``default behaviors'' of some of the layers, which may be overridden in special circumstances that we explicitly describe as such.
Let $x \in \B^\Z$ be arbitrary, and denote $y = f(x)$.
\begin{enumerate}
\item
  \label{item:rulemain}
  If $\pi_{\mathrm{main}}(x_1) = a \in \A_1$, or $\pi_{\mathrm{main}}(x_1) \in \W_{\A_1}$ and $\pi_{\mathrm{belt}}(x_0) = a \in \A_1$, or $\pi_{\mathrm{comp}}(x_0) = \#$ and $\pi_{\mathrm{main}}(x_1) = a \in \A_1$, then $\pi_{\mathrm{main}}(y_0) = a$.
  This means the $\A$-part of the Main Layer is generally shifted to the left.
  If the right neighbor of a cell is a wall, the data is instead copied from the Conveyor Belt Layer of the cell itself, onto the Main Layer of the same cell.
  Finally, if the cell has blank Computation Layer but its right neighbor does not, then the data is copied from the Conveyor Belt layer of that neighbor.
\item
  \label{item:ruleconv}
  If $\pi_{\mathrm{belt}}(x_{-1}) = a \in \A_1$, or $\pi_{\mathrm{comp}}(x_{-1}) = \#$ and $\pi_{\mathrm{main}}(x_0) = a \in \A_1$, then $\pi_{\mathrm{belt}}(y_0) = a$.
  This means the Conveyor Belt Layer is generally shifted to the right, and if the left neighbor of a cell has blank Computation Layer, the data is instead copied from the Main Layer of the cell itself, onto the Conveyor Belt Layer of the same cell.
\item
  \label{item:rulecomp}
  Suppose $\pi_{\mathrm{cmpr}}(x_0)$ is defined.
  If $\pi_{\mathrm{cmpr}}(x_1) = a \in \A_1$, or $\pi_{\mathrm{main}}(x_1) = \W_a$, then $\pi_{\mathrm{cmpr}}(y_0) = a$.
  This means the Comparison Layer is generally shifted to the left, and if the right neighbor of a cell is a wall, its decoration is copied instead.
\item
  \label{item:rulewall}
  Suppose $\pi_{\mathrm{main}}(x_0) \in \{\I\} \cup \W_{\A_1}$.
  If $\pi_{\mathrm{main}}(x_1) = a \in \A_1$, then $\pi_{\mathrm{main}}(y_0) = \W_a$.
  Otherwise $\pi_{\mathrm{main}}(y_0) = \W_\$$.
  This means all walls copy their decorations from the Main Layer of their right neighbor.
\item
  \label{item:ruleclock}
  
  If $\pi_{\mathrm{clock}}(x_{-1})$ is defined and not equal to $3$, then $\pi_{\mathrm{clock}}(y_0) = c'$, and otherwise $\pi_{\mathrm{clock}}(y_0) = c'+1$, where $3' = 0$ and $c' = c$ for $c \in \{0,1,2\}$.
\end{enumerate}

Items \ref{item:rulemain} and \ref{item:ruleconv} imply that if $x$ contains a length-$n$ contiguous run of cells whose Computational Layer is not blank, and which is bordered by a wall on the right and any symbol $b \in \B$ with $\pi_{\mathrm{comp}}(b) = \#$ on the left, then the Main and Conveyor Belt Layers of these cells hold a circular word $w \in \A_1^{2n}$ that $f$ continually rotates.
We call such a run of cells a \emph{conveyor belt}; each properly formatted segment will contain one.
From the last part of \cref{item:rulemain}, the symbols of $w$ are also copied on the Main Layer of the left bordering cell, which will thus receive a periodic sequence of symbols $w w w \cdots$.

Items \ref{item:rulecomp} and \ref{item:rulewall} imply that if a run of $\A_1$-cells 
on the Comparison Layer 
is bordered on the left by a wall 
, then that wall will capture the $\A_1$-symbols that are shifted toward it, and pass them to the Comparison Layers of the cells on its left.
This flow of information is depicted in Figure~\ref{fig:segments}.

The idea of item \ref{item:ruleclock} is that the Clock Layers of a finite run of cells encode a ternary counter that a single application of $f$ increments.
The least significant digit is the leftmost one, and the state $3$ denotes a $0$ that holds a carry. Carries propagate to the right.
The relevant property of the counter is the following.

\begin{lemma}
  \label{lem:BinaryCounter}
  Let $x \in \B^\Z$ be a configuration, and let $i \leq k \in \Z$ and $T \geq 0$ be such that for all $t \leq T$, $\pi_{\mathrm{clock}}(f^t(x_j)) \in \{0,1,2,3\}$ for each $i \leq j \leq k$, but $\pi_{\mathrm{clock}}(f^t(x_{i-1}))$ is undefined.
  Then the sequence $(\pi_{\mathrm{clock}}(f^t(x))_k)_{t = 0}^T$ is eventually periodic with transient part of length at most $k-i+1$ and eventual period of length $3^{k-i+1}$, and the state $3$ occurs exactly once every $3^{k-i+1}$ steps in the periodic tail.
\end{lemma}

\begin{proof}
  By induction on $k-i$.
  For $k = i$, the sequence has eventually periodic part $1, 2, 3, 1, 2, 3, 1, 2, 3, \ldots$ which is reached after at most $1 = k-i+1$ step (on which the state might be $0$).
  For $k > i$, we know the sequence of digits at position $k-1$ has eventual period of length $3^{k-i}$ and transient part of length at most $k-i$, and the state $3$ occurs at position $k-1$ exactly once in each period.
  The step after it does, the state at position $k$ is incremented by one, and if its new value is $3$, on the next step it resets to $0$. 
  On other time steps it retains its value.
  The claim follows.
\end{proof}

The idea of the construction is to write periodic configurations of (SFT approximations of) $X$ onto the conveyor belts, which in turn feed them to $\A_1$-regions.
All belts will eventually disappear from a generic configuration, leaving only the $\A_1$-regions whose contents approximate $X$ in the generic limit set of $f$.
The Comparison Layer captures this data through permeable walls, and the Turing Machine Layer analyzes it in order to control the merge process of segments by comparing the contents of two adjacent segments.

The Cleaning Layer $\B_{\mathrm{clean}}$ behaves exactly as in \cref{sec:gener-constr}, dividing the initial configuration into non-overlapping segments.
In particular, it retains the property that the outer signals $s_o$ erase all non-$s_o$ symbols they encounter, so that every segment initialized by $\I$-symbols is eventually fully formatted.
When an $\I$-symbol becomes a (decorated) wall on the first time step, it also produces a simulated head of the machine $\M$ in state $q_0$ on the Turing Machine Layer of its left neighbor.
In the following sections we describe how the machine performs computation and modifies the data on its segment.

\subsection{Computation of periodic points}

Under $f$, each formatted segment $S$ goes through four different stages, in the following order: \emph{computation stage}, \emph{waiting stage}, \emph{probe stage}, and \emph{merge stage}.
During the computation stage, the machine $\M$ computes and stores a periodic point of one of the SFTs given by \cref{lem:Pi02SFTApprox} or \cref{lem:Delta02SFTApprox}.
Once it is stored and continuously generated on the conveyor belt, the waiting stage begins. It lasts until the neighboring segment $S'$ on the right of $S$ has finished its computation stage.
In the probe stage, the machine reads and analyzes the periodic point stored by $S'$ to determine whether $S$ should merge with it.
Finally, in the merge stage the wall between the segments is erased and the periodic point of $S$ is glued to the one of $S'$.

The four stages are mostly controlled by the simulated Turing machine $\M$, which we now describe.
It differs from a standard Turing machine in several respects: we allow its head to move 0, 1 or 2 tape cells in one computation step, and to freely modify the contents of all cells in the vicinity of the simulated read-write head.
Even though the simulated head is always on a cell that has a non-blank Computational Layer, it can modify the states of nearby cells that do not, in order to extend its computational tape (but it will never create new heads).

Recall that a simulated machine $\M$ is initialized on the right end of every properly initialized segment.
We only describe the behavior of $\M$ in this context, as only the contents of properly initialized segments will be visible in the generic limit set of $f$ -- the rest is erased in finite time by $s_o$ and $s_i$ signals described in \cref{sec:gener-constr}, here through the Cleaning Layer.
Let thus $S$ be a properly formatted segment in the $f$-trajectory of a configuration.

First, the head of $\M$ travels to the left end of the segment $S$, extending the Computational Layer.
Then it measures the length $\ell$ of $S$, computes the largest power of two $2^m < \ell/2$, and erases the Computational Layer of the $\ell-2^m$ leftmost cells of $S$.
These cells will remain in $\A_1$-states from this point on, and the remaining $2^m$ cells of $S$ will have non-blank Computational Layers until the segment merges with another one on its right.
In particular, the machine $\M$ is now limited to $2^m = \Theta(\ell)$ tape cells.
We call $m$ the \emph{rank} of the segment $S$.
See \cref{fig:segments} for a diagram of the structure of formatted segments.

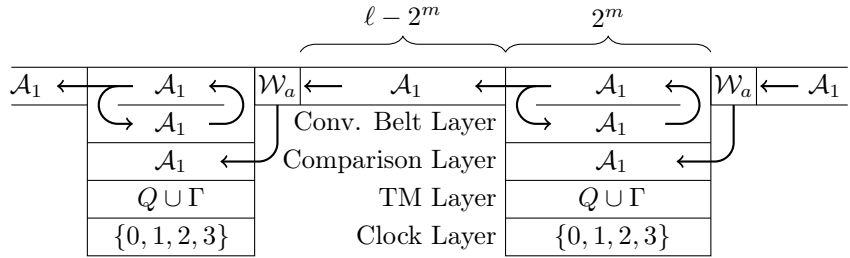
\begin{figure}[htp]
  \centering
  \begin{tikzpicture}[scale=1]

    
    \draw (3,0) -- (4,0);
    \draw (4.4,0) -- (5.8,0);
    \draw (6.2,0) -- (9.5,0);
    \draw (9.9,0) -- (11.8,0);
    \draw (12.2,0) -- (14,0);
    \draw (3,0.5) -- (14,0.5);

    \foreach \x in {
      5.6,11.6}{
      \draw [thick,->] (\x,-0.25) -- ++(0.2,0) arc (-90:90:0.25) -- ++(-0.2,0);
      \draw [thick,->] (\x+1.75,0.25) -- ++(-0.5,0);
      \draw [thick,->] (\x+0.9,0) -- ++(0,-0.5) arc (0:-90:0.25) -- ++(-0.5,0);
    }
    \foreach \x in {4.6,10.1}{
      \draw [thick,->] (\x-0.2,0.25) arc (90:270:0.25) -- ++(0.2,0);
      \draw [thick,->] (\x,0.25) -- ++(-1,0);
    }

    \foreach \x in {
      6,12}{
      \node at (\x+0.5,0.25) {$\W_a$};
      \draw (\x+0.2,0) -- +(0,0.5);
      \draw (\x+0.8,0) -- +(0,0.5);
    }

    \foreach \x/\xx in {
      4/6.2, 9.5/12.2}{
      \draw (\x,0.5) -- (\x,-2) -- (\xx,-2) -- (\xx,0.5);
      \draw (\x,-0.5) -- (\xx,-0.5);
      \draw (\x,-1) -- (\xx,-1);
      \draw (\x,-1.5) -- (\xx,-1.5);
    }

    \foreach \x/\y in {
      3.2/0.25,
      5.1/0.25, 5.1/-0.25, 5.1/-0.75,
      8.2/0.25,
      10.9/0.25, 10.9/-0.25, 10.9/-0.75,
      13.7/0.25
    }{
      \node at (\x,\y) {$\A_1$};
    }

    \foreach \x in {
      5.1, 10.9}{
      \node at (\x,-1.25) {$Q \cup \Gamma$};
      \node at (\x,-1.75) {$\{0,1,2,3\}$};
    }

    \node [left] at (9.5,-0.25) {Conv. Belt Layer};
    \node [left] at (9.5,-0.75) {Comparison Layer};
    \node [left] at (9.5,-1.25) {TM Layer};
    \node [left] at (9.5,-1.75) {Clock Layer};

    \draw [decorate,decoration={brace,amplitude=6pt}] (6.8,0.7) -- node [midway,above=6pt] {$\ell-2^m$} (9.5,0.7);
    \draw [decorate,decoration={brace,amplitude=6pt}] (9.5,0.7) -- node [midway,above=6pt] {$2^m$} (12.2,0.7);
    
  \end{tikzpicture}
  \caption{The anatomy of segments. Arrows indicate flow of information.}
  \label{fig:segments}
\end{figure}

Next, the machine $\M$ computes a word $u \in \A^{2^{m+1}}$ and stores it on the conveyor belt of the segment $S$.
The definition of $u$ is the first place where the two cases of the construction differ.
In the $\Pi^0_1 \subset \Pi^0_2$ case, $\M$ computes the triple $(X_m, Y_m, w_m)$ given by \cref{lem:Pi02SFTApprox}, which is doable in space $2^m$ if $\ell$ is large enough.
Here $X_m$ is a mixing SFT, $Y_m \subset X_m$ a nonempty SFT and $w_m \in \lang(X_m)$ a word of length $o(\log m)$.
The mixing distance and window size of $X_m$, and the window size of $Y_m$, are likewise $o(\log m)$.
Denote by $n_m$ the maximum of these numbers.
In the $\Delta^0_2$ case, $\M$ instead computes the pair $(X_m, w_m)$ given by \cref{lem:Delta02SFTApprox}, and we denote by $n_m = o(\log m)$ the maximum of the mixing distance and window size of $X_m$.
In both cases we may assume that the sequence $(n_m)_{m \in \N}$ is nondecreasing and $n_m \to \infty$ as $m \to \infty$.

By \cref{lem:LeastPeriod}, if $\ell$ is large enough, there exists a word $u \in \A^{2^{m+1}}$ such that the periodic configuration ${}^\infty u^\infty$ is in $X_m$, has least period $2^{m+1}$, contains an occurrence of $w_m$, and in the $\Pi^0_1 \subset \Pi^0_2$ case, contains an occurrence of some word $v_m \in \lang(Y_m)$ of length $n_m$. Indeed, we apply \cref{lem:LeastPeriod} to either $W = \{w_m\}$ or $W = \{w_m,v_m\}$, with $v_m$ having negligible length compared to $w_m$.

In the case where $\ell$ is not large enough for all of the above, we use $u = a^{2^{m+1}}$ for an arbitrary $a \in \A$ instead.

The machine $\M$ computes such a $u$ and writes it onto the conveyor belt.
This concludes the computation stage of $S$.

Under the CA $f$, the word $u$ is continually fed to the $\A_1$-cells on the left half of the segment $S$.
From this point on, these cells will always hold $\A$-states, that are $\A_1 \setminus \$$-states.

\subsection{Comparing periodic points}

When the machine $\M$ has finished writing the word $u \in \A^{2^{m+1}}$ onto the conveyor belt of its segment $S$, it initiates the waiting stage by traveling to the right end of $S$.
It waits there until the wall on its right stores an $\A$-state indicating that the segment $S'$ directly to the right of $S$ has finished its computation stage and stored some word $u' \in \A^*$ on its belt.
Once this happens, the segment $S$ enters the probe stage.

During the probe stage, the machine $\M$ will repeatedly capture a word occurring in the periodic point $x' = {}^\infty (u')^\infty$.
Note that on each time step, the Comparison Layer of $S$ now contains a length-$2^m$ subword of $x'$, which is continually shifted to the left and renewed through the wall between the two segments.
The machine $\M$ waits on the rightmost cell of $S$ until the Clock Layer of that cell contains a $3$.
We call this a \emph{clock signal}, and by \cref{lem:BinaryCounter}, it happens exactly once every $3^{2^m}$ time steps.
Then the machine repeatedly stores four adjacent symbols from the Comparison Layer onto a single cell of its computation tape, waits for three steps, and takes one step to the left.
Once it reaches the left end of the conveyor belt of $S$, its computation tape contains a word $v \in \A^{2^{m+2}}$ occurring in $x'$ (that is, $v$ is four times longer than the length of the computation tape).
This process is illustrated in \cref{fig:capture}.

\begin{figure}[htp]
  \centering
  \begin{tikzpicture}[xscale=1.3]

    \begin{scope}
      \draw (0,-0.5) -- (4.25,-0.5) -- (4.25,0.5);
      \draw (0,0.5) -- (8.25,0.5);
      \draw (0,1) -- (8.25,1);

      \node at (4,0.25) {$h$};
      \node at (4,0.75) {$a_0$};
      \node at (4.5,0.75) {$\W_{a_1}$};
      \node at (5,0.75) {$a_2$};
      \node at (5.5,0.75) {$a_3$};
      \node at (6,0.75) {$a_4$};
      \node at (6.5,0.75) {$a_5$};
      \node at (7,0.75) {$a_6$};
      \node at (7.5,0.75) {$a_7$};
      \node at (8,0.75) {$a_8$};
    \end{scope}

    \begin{scope}[yshift=2cm]
      \draw (0,-0.5) -- (4.25,-0.5) -- (4.25,0.5);
      \draw (0,0.5) -- (8.25,0.5);
      \draw (0,1) -- (8.25,1);

      \node at (4,0.25) {$h$};
      \node at (4,0) {\scriptsize $a_0 a_1$};
      \node at (4,-0.25) {\scriptsize $a_2 a_3$};
      \node at (3.5,0.75) {$a_0$};
      \node at (4,0.75) {$a_1$};
      \node at (4.5,0.75) {$\W_{a_2}$};
      \node at (5,0.75) {$a_3$};
      \node at (5.5,0.75) {$a_4$};
      \node at (6,0.75) {$a_5$};
      \node at (6.5,0.75) {$a_6$};
      \node at (7,0.75) {$a_7$};
      \node at (7.5,0.75) {$a_8$};
      \node at (8,0.75) {$a_9$};
    \end{scope}

    \begin{scope}[yshift=4cm]
      \draw (0,-0.5) -- (4.25,-0.5) -- (4.25,0.5);
      \draw (0,0.5) -- (8.25,0.5);
      \draw (0,1) -- (8.25,1);

      \node at (3.5,0.25) {$h$};
      \node at (4,0) {\scriptsize $a_0 a_1$};
      \node at (4,-0.25) {\scriptsize $a_2 a_3$};
      \node at (1.5,0.75) {$a_0$};
      \node at (2,0.75) {$a_1$};
      \node at (2.5,0.75) {$a_2$};
      \node at (3,0.75) {$a_3$};
      \node at (3.5,0.75) {$a_4$};
      \node at (4,0.75) {$a_5$};
      \node at (4.5,0.75) {$\W_{a_6}$};
      \node at (5,0.75) {$a_7$};
      \node at (5.5,0.75) {$a_8$};
      \node at (6,0.75) {$a_9$};
      \node at (6.5,0.75) {$a_{10}$};
      \node at (7,0.75) {$a_{11}$};
    \end{scope}

    \begin{scope}[yshift=6cm]
      \draw (0,-0.5) -- (4.25,-0.5) -- (4.25,0.5);
      \draw (0,0.5) -- (8.25,0.5);
      \draw (0,1) -- (8.25,1);

      \node at (3.5,0.25) {$h$};
      \node at (4,0) {\scriptsize $a_0 a_1$};
      \node at (4,-0.25) {\scriptsize $a_2 a_3$};
      \node at (3.5,0) {\scriptsize $a_4 a_5$};
      \node at (3.5,-0.25) {\scriptsize $a_6 a_7$};
      \node at (1,0.75) {$a_0$};
      \node at (1.5,0.75) {$a_1$};
      \node at (2,0.75) {$a_2$};
      \node at (2.5,0.75) {$a_3$};
      \node at (3,0.75) {$a_4$};
      \node at (3.5,0.75) {$a_5$};
      \node at (4,0.75) {$a_6$};
      \node at (4.5,0.75) {$\W_{a_7}$};
      \node at (5,0.75) {$a_8$};
      \node at (5.5,0.75) {$a_9$};
      \node at (6,0.75) {$a_{10}$};
      \node at (6.5,0.75) {$a_{11}$};
      \node at (7,0.75) {$a_{12}$};
    \end{scope}

  \end{tikzpicture}
  \caption{Capturing a word from the Comparison Layer. Time increases upward, possibly several steps at a time. Irrelevant layers and symbols are not shown. The letter $h$ represents the head of the Turing Machine.}
  \label{fig:capture}
\end{figure}
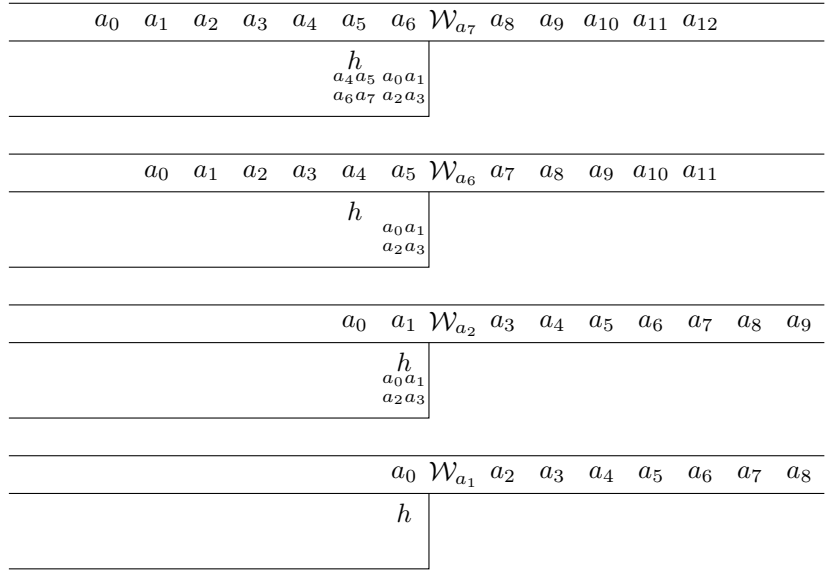

The machine $\M$ then checks whether the word $v$ is $q$-periodic for some $q \leq 2^{m+1}$.
If $v$ is not $q$-periodic for any $q \leq 2^{m+1}$, we say $\M$ has \emph{detected a merge candidate}.
The idea is that we want to merge $S$ with the segment $S'$ only if $S'$ has strictly higher rank, and detecting a merge candidate is evidence of this, since -- with the exception of ``false positives'' mentioned later -- having a larger period for the word in $S'$ means its conveyor belt itself was larger.
Detecting one merge candidate is not enough: once $\M$ has performed this analysis, it erases $v$ from its tape and starts over, waiting on the right end of the segment $S$ for another clock signal.
If $\ell$ is large enough, the capture, analysis and erasure of $v$ takes less than $3^{2^m}$ computation steps, and by handling short segments separately (using specific local rules with big enough radius), we may assume this is the case for all $\ell$.
Thus $\M$ can start capturing a new word every time the clock signal occurs.
The capturing process repeats until $\M$ has detected $m$ merge candidates in total (not necessarily consecutively), after which $S$ enters the merge stage.
The reason for this is that if $S'$ is produced by two short segments (of rank at most $m$) merging, right after this merge its Main Layer consists of two long periodic words separated by a short period breaker word (see the merge process in \cref{subsec:merge}).
The machine $\M$ will detect at most $m-1$ such ``false positive'' merge candidates in the worst case, see 
\cref{lem:MergeIsCorrect}.

\begin{lemma}
  \label{lem:EventualComparison}
  Suppose the segment $S'$ to the right of $S$ has rank $m' > m$.
  Eventually either $S$ enters the merge stage or $S'$ merges with another segment on its right.
\end{lemma}

\begin{proof}
  Let $u' \in \lang_{2^{m'+1}}(X_{m'})$ be the word stored on the conveyor belt of $S'$.
  By construction, the least period of the periodic point $x' = {}^\infty (u')^\infty \in X_{m'}$ is $2^{m'+1}$.
  The symbols of $x$ are shifted to the left on the Main Layer of $S'$, then through the wall separating $S$ and $S'$ onto the Comparison Layer of $S$.
  At each large enough time step $t$, if the segment $S'$ has not yet merged with another segment on its right, the rightmost symbol of the Comparison Layer of $S$ equals $x'_{i+t}$ for some initial offset $i$.
  The clock signal of $S$ arrives at time steps $t = j + n 3^{2^m}$ for $n \in \N$ and some initial offset $j$, at which point the machine $\M$ simulated in $S$ starts capturing a word of length $2^{m+2}$, which thus equals $v(n) := x'_{[i+j + n 3^{2^m}, i+j + n 3^{2^m} + 2^{m+2}-1]}$.
  Since $\gcd(3^{2^m}, 2^{m'+1}) = 1$, we have $\{v(n) \mid n \in \N\} = \{x'_{[n, n+2^{m+2}-1]} \mid n \in \Z\}$.
  If all of these subwords are periodic with period at most $2^{m+1}$, then so is $x'$ by Lemma~\ref{lem:PeriodicSubwords}, contradicting its construction.
  Hence at least one of the $v(n)$ is not periodic with a small period.
  When $\M$ captures this word, it detects a merge candidate, and when it has done so $m$ times, $S$ enters the merge stage.
\end{proof}

\subsection{Merging segments}
\label{subsec:merge}

We now describe the merge stage of the segment $S$.
Here the two cases differ more substantially.
Recall that $n_m = o(\log m)$ is a window size for $X_m$ and $Y_m$, and a mixing distance for $X_m$.


In the $\Delta^0_2$ case, $\M$ captures a word $u \in \A^{n_m}$ from the Main Layer of $S$ onto its computation tape, then rewrites the $n_m$ symbols to the right of $u$ on the Comparison Layer with $\$$-symbols, and finally captures another word $w \in \A^{n_m}$ from the Comparison Layer that occurs after the rewritten symbols.
This process is controlled by some auxiliary markings $\M$ placed at the beginning of the merge stage; we omit the exact implementation details.
See \cref{fig:merge} for an illustration.
We use the $\$$-symbols to mark the cells between $u$ and $w$ so that $\M$ can find them later; recall that the Main and Comparison Layers are continually shifted to the left by $f$.

In the $\Pi^0_1 \subset \Pi^0_2$ case, $\M$ waits for a clock signal before capturing the words $u$ and $w$.
Then it checks whether $w \in \lang(Y_m)$, which takes $\exp(O(n_m)) = m^{o(m)}$ computation steps.
If this is not the case, then $\M$ erases the words $u$ and $w$ from its tape, waits for another clock signal, and repeats the capturing process.

\begin{figure}[htp]
  \centering
  \begin{tikzpicture}[xscale=1.3]

    \begin{scope}
      \draw (1,0.5) -- (10,0.5);
      \draw (1,0) -- (7,0);
      \draw (7.5,0) -- (10,0);
      \draw (1,-0.5) -- (7.5,-0.5);
      \draw (1,-1) -- (7.5,-1);
      \draw (1,-2) -- (7.5,-2) -- (7.5,0);

      \node at (7.25,0.25) {$u_0$};
      \node at (7.25,-0.25) {$u_1$};
      \node at (6.75,-0.25) {$u_2$};
      \node at (6.25,-0.25) {$u_3$};

      \node at (7.75,0.25) {$\W$};
      

      \node at (7.25,-1.25) {$h$};

      \node at (6.75,-1.75) {$\bullet$};
      \node at (4.25,-1.75) {$\bullet$};
      \node at (4.75,-1.75) {$\bullet$};
    \end{scope}

    \begin{scope}[yshift=3.5cm]
      \draw (1,0.5) -- (10,0.5);
      \draw (1,0) -- (7,0);
      \draw (7.5,0) -- (10,0);
      \draw (1,-0.5) -- (7.5,-0.5);
      \draw (1,-1) -- (7.5,-1);
      \draw (1,-2) -- (7.5,-2) -- (7.5,0);

      \node at (5.25,0.25) {$u_0$};
      \node at (5.75,0.25) {$u_1$};
      \node at (6.25,0.25) {$u_2$};
      \node at (6.75,0.25) {$u_3$};

      \node at (7.75,0.25) {$\W$};
      
      \node at (9.25,0.25) {$w_0$};
      \node at (9.75,0.25) {$w_1$};

      \node at (6.75,-1.25) {$h$};
      \node at (7.25,-1.75) {\scriptsize $u_0 u_1$};
      \node at (6.75,-1.75) {\scriptsize $u_2 u_3$};
      \node at (4.25,-1.75) {$\bullet$};
      \node at (4.75,-1.75) {$\bullet$};
    \end{scope}

    \begin{scope}[yshift=7cm]
      \draw (1,0.5) -- (10,0.5);
      \draw (1,0) -- (7,0);
      \draw (7.5,0) -- (10,0);
      \draw (1,-0.5) -- (7.5,-0.5);
      \draw (1,-1) -- (7.5,-1);
      \draw (1,-2) -- (7.5,-2) -- (7.5,0);

      \node at (1.25,0.25) {$u_1$};
      \node at (1.75,0.25) {$u_2$};
      \node at (2.25,0.25) {$u_3$};

      \node at (7.75,0.25) {$\W$};

      \node at (2.75,-0.75) {$\$$};
      \node at (3.25,-0.75) {$\$$};
      \node at (3.75,-0.75) {$\$$};
      \node at (4.25,-0.75) {$\$$};
      
      \node at (4.75,-0.75) {$w_0$};
      \node at (5.25,-0.75) {$w_1$};
      \node at (5.75,-0.75) {$w_2$};
      \node at (6.25,-0.75) {$w_3$};

      \node at (4.75,-1.25) {$h$};
      \node at (7.25,-1.75) {\scriptsize $u_0 u_1$};
      \node at (6.75,-1.75) {\scriptsize $u_2 u_3$};
      \node at (4.25,-1.75) {$\bullet$};
    \end{scope}

    \begin{scope}[yshift=10.5cm]
      \draw (1,0.5) -- (10,0.5);
      \draw (1,0) -- (7,0);
      \draw (7.5,0) -- (10,0);
      \draw (1,-0.5) -- (7.5,-0.5);
      \draw (1,-1) -- (7.5,-1);
      \draw (1,-2) -- (7.5,-2) -- (7.5,0);


      \node at (7.75,0.25) {$\W$};

      \node at (1.25,-0.75) {$\$$};
      \node at (1.75,-0.75) {$\$$};
      \node at (2.25,-0.75) {$\$$};
      
      \node at (2.75,-0.75) {$w_0$};
      \node at (3.25,-0.75) {$w_1$};
      \node at (3.75,-0.75) {$w_2$};
      \node at (4.25,-0.75) {$w_3$};

      \node at (4.25,-1.25) {$h$};
      \node at (7.25,-1.75) {\scriptsize $u_0 u_1$};
      \node at (6.75,-1.75) {\scriptsize $u_2 u_3$};
      \node at (4.75,-1.75) {\scriptsize $w_0 w_1$};
      \node at (4.25,-1.75) {\scriptsize $w_2 w_3$};
    \end{scope}

  \end{tikzpicture}
  \caption{Capturing words at the beginning of the merge stage, illustrated with $n_m = 4$. Time increases upward several steps at a time. Irrelevant symbols and layers are not shown. The letter $h$ represents the head of the Turing Machine. The dots are the auxiliary markings $\M$ has placed beforehand.}
  \label{fig:merge}
\end{figure}
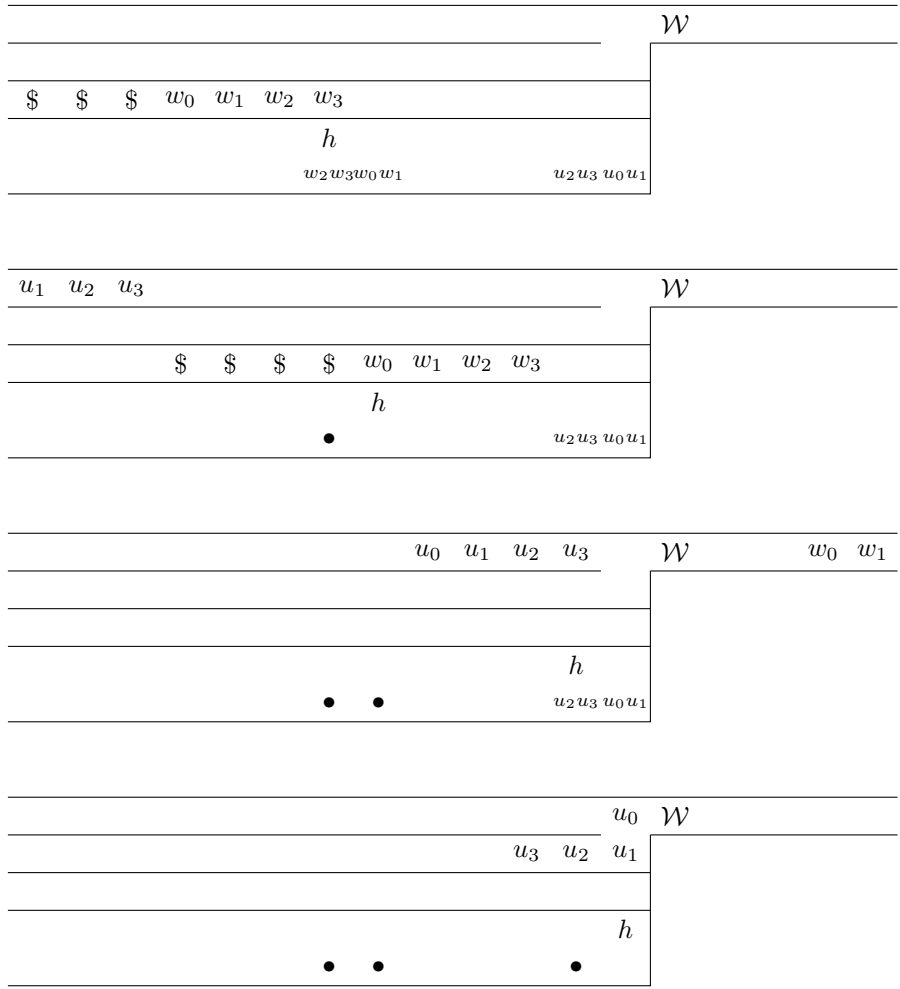

Next (immediately after capturing $u$ and $w$ in the $\Delta^0_2$ case, and as soon as $w \in \lang(Y_m)$ in the $\Pi^0_1 \subset \Pi^0_2$ case), $\M$ computes a \emph{merge gluing word} $v \in \A^{n_m}$ as follows.
In the $\Pi^0_1 \subset \Pi^0_2$ case, we simply require that $u v w \in \lang(X_m)$.
Such a word exists since $u \in \lang(X_m)$ and $w \in \lang(Y_m) \subset \lang(X_m)$, and $n_m$ is a mixing distance for $X_m$.
By \cref{lem:GluingTimeComplexity}, $\M$ can compute $v$ in $\exp(O(n_m)) = m^{o(m)}$ steps.
In the $\Delta^0_2$ case, $\M$ computes the largest integer $0 \leq d \leq n_m$ such that the length-$d$ prefix $w_{[0, d-1]}$ occurs in the SFT approximation $X_{d,m} := \mathcal{S}_d(X_m)$.
Note that each $X_{d,m}$ is also mixing with mixing distance $n_m$.
Then it finds a $v$ such that $u v w_{[0, d-1]} \in \lang(X_{d,m})$, again in $m^{o(m)}$ steps.

The rest of the merge process is identical for the two cases.
The machine modifies the conveyor belt of $S$ by replacing the $\$$-symbols on the Comparison Layer with the symbols of $v$, one by one.
As the $\$$-symbols are now within distance $m^{o(m)}$ from the right end of $S$ and traveling left with constant speed, the $i$th symbol takes $\exp(O(i)) \cdot m^{o(m)}$ steps to replace, for a total of $\exp(m^{o(m)})$ steps.
For large enough $m$ we have $\exp(m^{o(m)}) \ll 2^m$, so there is enough time for $\M$ to perform these operations before the $\$$-symbols reach the left end of the conveyor belt of $S$.
By handling short conveyor belts separately, we may assume this applies to all segments.
After this, the Main Layer of the conveyor belt contains a word of the form $a u v b$, and the Comparison Layer contains $c w d$, such that $|a u v| = |c|$.
We may arrange the copying process so that the simulated head of $\M$ ends up on top of the leftmost symbol of $v$.

Next, $\M$ travels left at speed 1 together with the Main and Comparison Layers.
When it hits the left end of the conveyor belt, it turns back to the right and erases the Computation Layer of the segment $S$; it also rewrites the Main Layer with the contents of the Comparison Layer. Consequently, $auvwd$ ends up printed on the Main Layer, and the Computation Layer shrinks by ``retracting'' to the right of its segment.
This process is illustrated in \cref{fig:merge2}.
When $\M$ reaches the right end of $S$, it erases itself and replaces the wall $\W_a$ with its decoration $a$ as well. 
In this way, the segment $S$ merges with its neighbor $S'$ into one longer segment whose conveyor belt is identical to that of $S'$.
This concludes the definition of $f$.

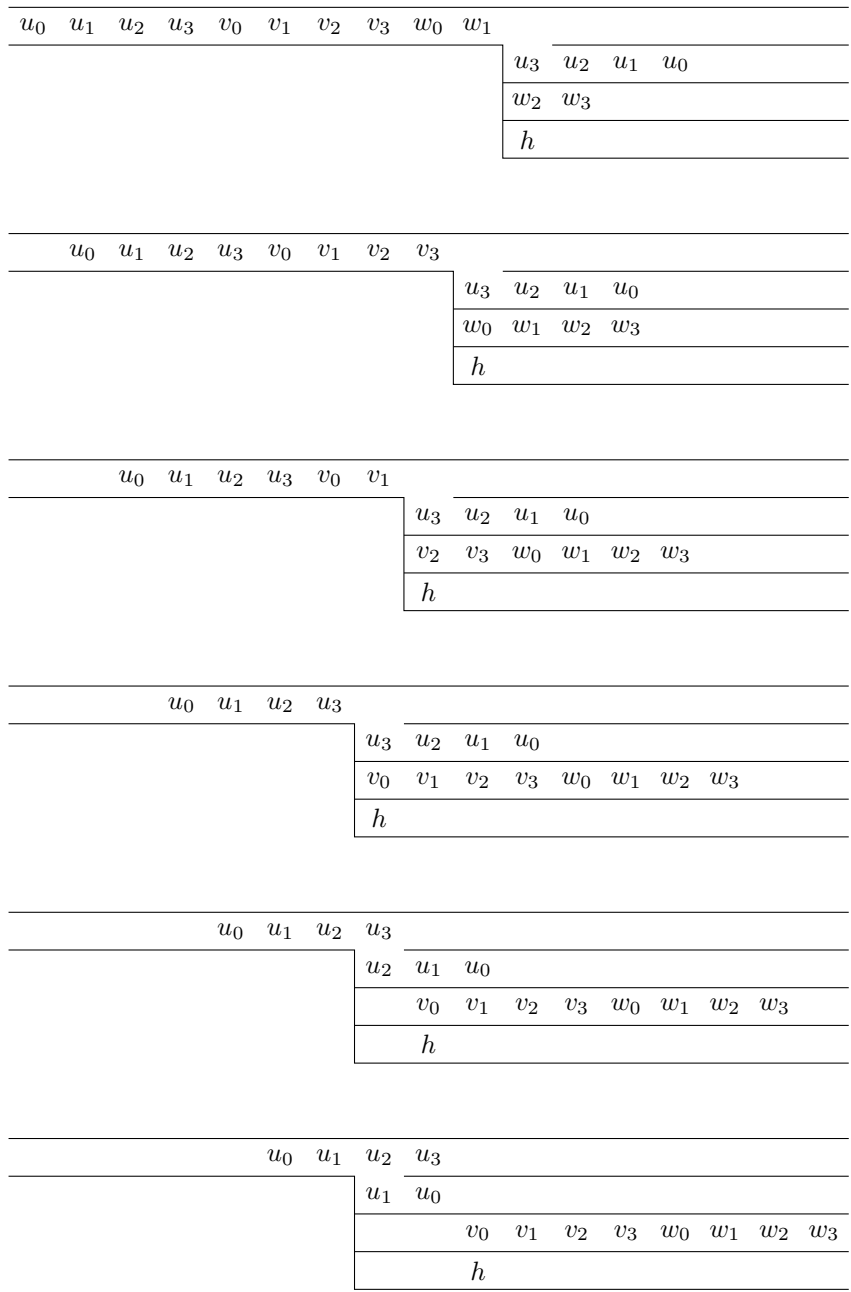
\begin{figure}[htp]
  \centering
  \begin{tikzpicture}[xscale=1.3]

    \begin{scope}
      \draw (2.5,0.5) -- (11,0.5);
      \draw (2.5,0) -- (6,0) -- (6,-1.5) -- (11,-1.5);
      \draw (6.5,0) -- (11,0);
      \draw (6,-0.5) -- (11,-0.5);
      \draw (6,-1) -- (11,-1);

      \node at (5.25,0.25) {$u_0$};
      \node at (5.75,0.25) {$u_1$};
      \node at (6.25,0.25) {$u_2$};
      \node at (6.75,0.25) {$u_3$};

      \node at (6.25,-0.25) {$u_1$};
      \node at (6.75,-0.25) {$u_0$};

      \node at (7.25,-0.75) {$v_0$};
      \node at (7.75,-0.75) {$v_1$};
      \node at (8.25,-0.75) {$v_2$};
      \node at (8.75,-0.75) {$v_3$};

      \node at (9.25,-0.75) {$w_0$};
      \node at (9.75,-0.75) {$w_1$};
      \node at (10.25,-0.75) {$w_2$};
      \node at (10.75,-0.75) {$w_3$};

      \node at (7.25,-1.25) {$h$};
    \end{scope}

    \begin{scope}[yshift=3cm]
      \draw (2.5,0.5) -- (11,0.5);
      \draw (2.5,0) -- (6,0) -- (6,-1.5) -- (11,-1.5);
      \draw (6.5,0) -- (11,0);
      \draw (6,-0.5) -- (11,-0.5);
      \draw (6,-1) -- (11,-1);

      \node at (4.75,0.25) {$u_0$};
      \node at (5.25,0.25) {$u_1$};
      \node at (5.75,0.25) {$u_2$};
      \node at (6.25,0.25) {$u_3$};

      \node at (6.25,-0.25) {$u_2$};
      \node at (6.75,-0.25) {$u_1$};
      \node at (7.25,-0.25) {$u_0$};

      \node at (6.75,-0.75) {$v_0$};
      \node at (7.25,-0.75) {$v_1$};
      \node at (7.75,-0.75) {$v_2$};
      \node at (8.25,-0.75) {$v_3$};

      \node at (8.75,-0.75) {$w_0$};
      \node at (9.25,-0.75) {$w_1$};
      \node at (9.75,-0.75) {$w_2$};
      \node at (10.25,-0.75) {$w_3$};

      \node at (6.75,-1.25) {$h$};
    \end{scope}

    \begin{scope}[yshift=6cm]
      \draw (2.5,0.5) -- (11,0.5);
      \draw (2.5,0) -- (6,0) -- (6,-1.5) -- (11,-1.5);
      \draw (6.5,0) -- (11,0);
      \draw (6,-0.5) -- (11,-0.5);
      \draw (6,-1) -- (11,-1);

      \node at (4.25,0.25) {$u_0$};
      \node at (4.75,0.25) {$u_1$};
      \node at (5.25,0.25) {$u_2$};
      \node at (5.75,0.25) {$u_3$};

      \node at (6.25,-0.25) {$u_3$};
      \node at (6.75,-0.25) {$u_2$};
      \node at (7.25,-0.25) {$u_1$};
      \node at (7.75,-0.25) {$u_0$};

      \node at (6.25,-0.75) {$v_0$};
      \node at (6.75,-0.75) {$v_1$};
      \node at (7.25,-0.75) {$v_2$};
      \node at (7.75,-0.75) {$v_3$};

      \node at (8.25,-0.75) {$w_0$};
      \node at (8.75,-0.75) {$w_1$};
      \node at (9.25,-0.75) {$w_2$};
      \node at (9.75,-0.75) {$w_3$};

      \node at (6.25,-1.25) {$h$};
    \end{scope}

    \begin{scope}[yshift=9cm]
      \draw (2.5,0.5) -- (11,0.5);
      \draw (2.5,0) -- (6.5,0) -- (6.5,-1.5) -- (11,-1.5);
      \draw (7,0) -- (11,0);
      \draw (6.5,-0.5) -- (11,-0.5);
      \draw (6.5,-1) -- (11,-1);

      \node at (3.75,0.25) {$u_0$};
      \node at (4.25,0.25) {$u_1$};
      \node at (4.75,0.25) {$u_2$};
      \node at (5.25,0.25) {$u_3$};

      \node at (6.75,-0.25) {$u_3$};
      \node at (7.25,-0.25) {$u_2$};
      \node at (7.75,-0.25) {$u_1$};
      \node at (8.25,-0.25) {$u_0$};

      \node at (5.75,0.25) {$v_0$};
      \node at (6.25,0.25) {$v_1$};
      \node at (6.75,-0.75) {$v_2$};
      \node at (7.25,-0.75) {$v_3$};

      \node at (7.75,-0.75) {$w_0$};
      \node at (8.25,-0.75) {$w_1$};
      \node at (8.75,-0.75) {$w_2$};
      \node at (9.25,-0.75) {$w_3$};

      \node at (6.75,-1.25) {$h$};
    \end{scope}

    \begin{scope}[yshift=12cm]
      \draw (2.5,0.5) -- (11,0.5);
      \draw (2.5,0) -- (7,0) -- (7,-1.5) -- (11,-1.5);
      \draw (7.5,0) -- (11,0);
      \draw (7,-0.5) -- (11,-0.5);
      \draw (7,-1) -- (11,-1);

      \node at (3.25,0.25) {$u_0$};
      \node at (3.75,0.25) {$u_1$};
      \node at (4.25,0.25) {$u_2$};
      \node at (4.75,0.25) {$u_3$};

      \node at (7.25,-0.25) {$u_3$};
      \node at (7.75,-0.25) {$u_2$};
      \node at (8.25,-0.25) {$u_1$};
      \node at (8.75,-0.25) {$u_0$};

      \node at (5.25,0.25) {$v_0$};
      \node at (5.75,0.25) {$v_1$};
      \node at (6.25,0.25) {$v_2$};
      \node at (6.75,0.25) {$v_3$};

      \node at (7.25,-0.75) {$w_0$};
      \node at (7.75,-0.75) {$w_1$};
      \node at (8.25,-0.75) {$w_2$};
      \node at (8.75,-0.75) {$w_3$};

      \node at (7.25,-1.25) {$h$};
    \end{scope}

    \begin{scope}[yshift=15cm]
      \draw (2.5,0.5) -- (11,0.5);
      \draw (2.5,0) -- (7.5,0) -- (7.5,-1.5) -- (11,-1.5);
      \draw (8,0) -- (11,0);
      \draw (7.5,-0.5) -- (11,-0.5);
      \draw (7.5,-1) -- (11,-1);

      \node at (2.75,0.25) {$u_0$};
      \node at (3.25,0.25) {$u_1$};
      \node at (3.75,0.25) {$u_2$};
      \node at (4.25,0.25) {$u_3$};

      \node at (7.75,-0.25) {$u_3$};
      \node at (8.25,-0.25) {$u_2$};
      \node at (8.75,-0.25) {$u_1$};
      \node at (9.25,-0.25) {$u_0$};

      \node at (4.75,0.25) {$v_0$};
      \node at (5.25,0.25) {$v_1$};
      \node at (5.75,0.25) {$v_2$};
      \node at (6.25,0.25) {$v_3$};

      \node at (6.75,0.25) {$w_0$};
      \node at (7.25,0.25) {$w_1$};
      \node at (7.75,-0.75) {$w_2$};
      \node at (8.25,-0.75) {$w_3$};

      \node at (7.75,-1.25) {$h$};
    \end{scope}
    
  \end{tikzpicture}
  \caption{Erasing the Computation Layer, illustrated with $n_m = 4$. Time increases upward. Irrelevant states and layers are not shown. The letter $h$ represents the head of the Turing Machine.}
  \label{fig:merge2}
\end{figure}

\subsection{Proof of correctness}

With $f$ defined as above, we claim that its generic limit set is exactly $X$.
\Cref{thm:MixingRealization} directly follows, since $f|_{\A^\Z} = \sigma|_{\A^\Z}$.
Before that, we prove a few more lemmas about the behavior of segments under $f$.

\begin{lemma}
  \label{lem:EventualMerge}
  Let $S$ and $S'$ be neighboring segments such that $S'$ has higher rank than $S$, and suppose $S$ has entered the merge stage.
  Then eventually either $S$ merges with $S'$, or $S'$ initiates its own merge process.
\end{lemma}

Note that, \textit{a priori}, $S'$ might not have another segment on its right when it initiates the merge process.

\begin{proof}
  In the $\Delta^0_2$ case this is clear: once $S$ enters the merge stage, it will capture $u$ and $w$, compute the number $d$ and the associated merge glue word $v$, and merge the segments.

  Consider then the $\Pi^0_1 \subset \Pi^0_2$ case.
  Let $m < m'$ be the ranks of $S$ and $S'$, and let $u \in \lang_{2^{m+1}}(X_m)$ and $u' \in \lang_{2^{m'+1}}(x_{m'})$ be the words stored on their conveyor belts.
  By construction, $u'$ has a length-$n_{m'}$ subword $w = w_{m'} \in \lang(Y_{m'}) \subset \lang(Y_m)$.
  As in the proof of \cref{lem:EventualComparison}, the simulated machine $\M$ in $S$ repeatedly captures all of the subwords $\{ x'_{[n, n+n_m-1]} \mid n \in \Z \}$ of length $n_m$ in some order.
  Since $n_{m'} \geq n_m$, at least one of these words is a subword of $w$.
  Thus it occurs in $Y_m$ and causes $M$ to initiate the merge process, erasing the conveyor belt of $S$ and the wall between $S$ and $S'$, and rewriting its Main Layer as described above, unless $S'$ initiates its own merge process.
\end{proof}

\begin{lemma}
  \label{lem:MergeIsCorrect}
  Suppose a rank-$m$ segment $S$ has entered the merge stage.
  Then directly on its right there is another properly formatted segment of rank strictly above $m$.
\end{lemma}

\begin{proof}
  The proof is identical for the two cases.
  It is enough to prove the result in the case where $S$ has just entered the merge stage, since then we know that the rank of the segment directly to the right of $S$ can only increase with time.
  
  We proceed by induction on the time step $t \in \N$ on which $S$ enters the merge stage.
  When $S$ enters the merge stage, the simulated machine $\M$ in it has detected $m$ merge candidates, which are captured words in $\A^{2^{m+2}}$ that are not $q$-periodic for any $q \leq 2^{m+1}$.
  Since $S$ is a segment produced by the walls-and-counters construction, these words must originate from the conveyor belts of other segments.
  Thus, on some time step $t' < t$, there was another segment directly to the right of $S$.
  By the induction hypothesis, before time $t$ that segment could only merge with properly formatted segments of higher rank.
  Since the number of time steps it takes to format a segment grows with its length, this implies that at time $t$ there is also a segment $S'$ -- possibly not properly formatted -- directly to the right of $S$.
  Let $m'$ be its rank.
  
  A merge candidate captured by $S$ cannot be a subword of any periodic point stored in a conveyor belt of length at most $2^{m+1}$.
  Thus, either it originates from a subword of a periodic point stored on a longer conveyor belt, or at least one of its symbols originates from a merge glue word.
  If the former condition holds for even one of the merge candidates, then the rank of the segment $S''$ directly to the right of $S$ at some time step $t' < t$ was strictly above $m$.
  Since either $S' = S''$ or $S'$ is produced by $S''$ merging with some segments to its right, which have even higher ranks by the induction hypothesis, we have $m' > m$.

  Suppose then that all $m$ merge candidates contain symbols that originate from merge glue words.
  If any of these merge glue words is the result of a merge where one of the segments had rank above $m$, then $m' > m$ for the same reason as above.
  Suppose then that all of them result from merges between segments with rank at most $m$.
  Then the merge glue words have length at most $n_m = o(\log m)$.
  Since $\M$ starts the capture process only at clock signals, which occur every $3^{2^m}$ time steps, and $n_m < 2^{m+2} \ll 3^{2^m}$, two merge candidates cannot contain symbols originating from the same merge glue word -- at least as long as $m$ is big enough, and we can as always suppose that small rank cases are handled separately.
  Since a merge glue word appears uniquely with the merging of two segments, the segment $S'$ is the result of at least $m+1$ segments eventually merging into one.
  Since the lowest possible rank is $1$, by the induction hypothesis we have $m' > m$.
\end{proof}

\begin{lemma}
  \label{lem:LanguageConverges}
  For each $k \in \N$ there exists $m_k \in \N$ such that the following holds.
  Let $x \in \B^\Z$ be a configuration, $t \geq 0$, and $[i, j] \subset \Z$ an interval such that $S = f^t(x)_{[i,j]}$ is a segment of rank $m \geq m_k$ that has just entered the merge stage.
  Then for all $t' \geq t$ we have $f^{t'}(x)_{[j+1,j+k]} \in \lang(X)$.
\end{lemma}

\begin{proof}
  Recall that $n_m = o(\log m)$ is a window size for $X_m$ (and $Y_m$ in the case where it exists), and a mixing distance for $X_m$.
  There exists $m_k \geq 0$ such that $\log_2 m_k > k$, $n_{m_k} \geq k$ and for all $\ell \geq m_k$, the SFT $X_\ell$ produced by \cref{lem:Pi02SFTApprox} satisfies $\lang_k(X_\ell) \subset \lang_k(X)$, and the one produced by \cref{lem:Delta02SFTApprox} satisfies $\lang_k(X_\ell) = \lang_k(X)$.
        
  Fix $k \geq 0$.
  We prove by induction on $t'$ that $m_k$ has the required properties.
  Choose $m \geq m_k$, $x$, $t$ and $[i,j]$ as in the claim.
  \Cref{lem:MergeIsCorrect} implies that in $f^t(x)$, there is a segment $S'$ of rank $m' > m$ directly to the right of $S$.
  Since $\log_2 m_k > k$, this segment's $\A$-part contains the interval $[j+1,j+k]$.
  Applying the same lemma repeatedly (whenever the segment containing cell $j+1$ merges with another one), we see that for all $t'' \geq t$, there is a segment of rank at least $m' > m_k$ in $f^{t''}(x)$ whose $\A$-part contains $[j+1,j+k]$.
  
  Now, the word $r = f^{t'}(x)_{[j+1,j+k]}$ either originates from the conveyor belt of a segment of some rank $\ell$, or at least one of its symbols originates from a merge gluing word $v \in \A^{n_\ell}$, computed as part of $u v w \in \A^{3 n_\ell}$ during the merging of two segments of some ranks $\ell < \ell'$ (see \cref{subsec:merge}).
  The same holds for each of the $m$ merge candidates captured by $S$ before time $t$.
  By \cref{lem:MergeIsCorrect}, the last merge candidate originates from a segment of rank at least $m$, so we have $\ell \geq m \geq m_k$.

  There are a few cases to consider.
  \begin{itemize}
  \item
    The word $r$ originates from a conveyor belt.
    Then we have $r \in \lang_k(X_\ell) \subset \lang_k(X)$ by our choice of $m_k$.
  \item
    Some symbol of $r$ originates from a merge gluing word and we are in the $\Pi^0_1 \subset \Pi^0_2$ case.
    Then $u v w \in \lang(Y_\ell)$, hence $r \in \lang_k(Y_\ell) \subset \lang_k(X_\ell) \subset \lang_k(X)$ by our choice of $m_k$.
  \item
    Some symbol of $r$ originates from a merge gluing word and we are in the $\Delta^0_2$ case.
    Then $u v w_{[0, d-1]} \in \lang(\mathcal{S}_d(X_\ell))$ for the largest $0 \leq d \leq n_\ell$ with $w_{[0,d-1]} \in \lang(X_\ell)$.
    The word $w_{[0,k-1]}$ was captured by the segment of rank $\ell$ before time step $t'$.
    By the induction hypothesis applied to that segment and our choice of $m_k$, we have $n_\ell \geq n_{m_k} \geq k$ and $w_{[0,k-1]} \in \lang_k(X) = \lang_k(X_\ell)$.
    It follows that $d \geq k$ since $d$ is the largest possible integer so that $w_{[0,d-1]} \in \lang(X_\ell)$ holds, and then $r \in \lang_k(\mathcal{S}_d(X_\ell)) = \lang_k(X_\ell) = \lang_k(X)$.
  \end{itemize}
  In all cases we have $r \in \lang(X)$.
\end{proof}

\begin{proof}[Proof of \cref{thm:MixingRealization}]
  We claim that the CA $f$ constructed in this section satisfies $\gls(f) = X$.
  The theorem follows from this, since $f|_{\A^\Z} = \sigma|_{\A^\Z}$ by construction, and $X \subset \A^\Z$.
  
  Take any word $s \in \lang(X)$.
  It occurs infinitely many times as $w_m$ in the sequence of triples $(X_m, Y_m, w_m)$ given by \cref{lem:Pi02SFTApprox}, or in the sequence of pairs $(X_m, w_m)$ given by \cref{lem:Delta02SFTApprox}.
  Thus, in both cases of the theorem there are infinitely many different numbers $\ell$ such that a segment of length $\ell$ produced by the walls-and-counters part of $f$ stores on its conveyor belt a periodic point that contains an occurrence of $s$.

  We claim that the empty word enables $s$ in the sense of \cref{lemma2}.
  Take any cylinder set $[v]_i \subset \B^\Z$, where we may assume $|v| \geq 2 |s|$ and $-|v| < i \leq 0$.
  Choose a large integer $k \in \N$ and consider the configuration $x = {}^\infty \$ \I . v \I \$^k \I \$^\infty \in [v]_i$, where the dot denotes coordinate $i$.
  If $k$ is large enough, the Cleaning Layer guarantees that $f^k(x)_{[i-1, i+|v|]}$ is a sequence of segments separated by walls, and \cref{lem:EventualComparison} and \cref{lem:EventualMerge} guarantee that the length-$k$ segment on their right -- the one that starts as $\I\$^k\I$ -- will eventually merge with them all.
  This means that for large enough $t$, the word $f^t(x)_{[i-1, i+|v|+k+1]}$ is a single segment, and $f^t(x)_{[0, |s|-1]} \in \A^*$ lies in its $\A$-part. Due to the previous paragraph, we can find infinitely many $k$'s -- and consequently infinitely many $x$'s in $[v]_i$ -- so that the conveyor belt of the length-$k$ segment contains an occurrence of $s$.
  This shows that the empty word enables $s$, hence $s \in \lang(\gls(f))$.
  
  Conversely, let $s \in \lang(\gls(f))$ be arbitrary.
  By \cref{lemma2}, some cylinder set $[v]_i \subset \B^\Z$ enables it.
  We may again assume that $|v| \geq 2 |s|$ and $-|s| < i \leq 0$.
  Let $m = m_{|s|}$ be given by \cref{lem:LanguageConverges} for $k = |s|$, denote $\ell = 2^{m+1}+1$ and consider the words $u = \$^n \I \$^m \I$ and $w = \I \$^{4(\ell + |v|)} \I \$^n$ for some large $n \in \N$.
  Since $v$ enables $s$, we have that for infinitely many $t \in \N$, the cylinder set $C = [u v w]_{i-n-\ell-2}$ intersects $f^{-t}([s])$.
  
  For all $x \in C$, the word $x_{[i-\ell-2, i+4(\ell + |v|)+1]}$ is a sequence of segments, the leftmost of which has rank $m$ and the rightmost of which has the highest rank.
  \Cref{lem:EventualComparison} and \cref{lem:EventualMerge} guarantee that as long as $n$ is large enough (so that the left- and rightmost segments have no time to merge with any other segments), all these segments will eventually merge into one.
  Suppose this happens at time $t$.
  Let $t' < t$ be the time step at which a rank-$m$ segment whose right wall is at coordinate $i$ enters the merge stage.
  By our choice of $m = m_{|s|}$, we then have $f^{t''}(x)_{[i+1,i+k]} \in \lang_k(X)$ for all $t'' \geq t'$ and $x \in C$.
  There exist $t''' \geq t+i+1$ and $x \in C$ with $f^{t'''}(x)_{[0,k-1]} = s$ due to $v$ enabling $s$. Then that word $s$ has been shifted to the left during the $i+1$ previous time steps, and combining this with $f^{t'''-i-1}(x)_{[i+1,i+k]} \in \lang_k(X)$, we obtain that $f^{t'''-i-1}(x)_{[i+1,i+k]} = s \in \lang_k(X)$. Therefore $s \in \lang(X)$, which concludes the proof.
\end{proof}

\subsection{Corollaries}

In the situation of \cref{thm:MixingRealization}, \cref{cor:ShiftGivesMinimal} implies that $\gls(f)$ is inclusion-minimal.
In particular, as there exist chain mixing subshifts with $\Pi^0_2$-complete languages (for example, the subshift $X \subset \{0,1,2\}^\Z$ defined by forbidding $2 w 2$ for each $w \in \{0,1\}^* \setminus L$, where $L \subset \{0,1\}^*$ is $\Pi^0_2$-complete), such subshifts can be built through \cref{thm:MixingRealization}. Consequently, the complexity bound of \cref{pi2b} is optimal.

\begin{corollary}
	\label{cor:InclusionMinimalOptimalBound}
	There exists a CA $f$ with $\gls(f)$ an inclusion-minimal GLS, such that $\lang(\gls(f))$ is a $\Pi^0_2$-complete set.
\end{corollary}

We can also use \cref{thm:MixingRealization} to characterize generic limit sets among several classes of subshifts.

\begin{corollary}
  \label{cor:ChainTransCharacterization}
  Let $X \subset \A^\Z$ be a one-dimensional chain transitive subshift that is either $\Pi^0_2$ and contains a nonempty $\Pi^0_1$ subshift, or is $\Delta^0_2$.
  Then $X$ is a generic limit set of some CA if and only if it is chain mixing.
\end{corollary}

\begin{proof}
  If $X$ is chain mixing, then \cref{thm:MixingRealization} implies that it can be realized as a generic limit set.
  Otherwise, \cref{lem:Akin} and \cref{coro:chaintrans} show that $X$ is not a generic limit set.
\end{proof}

For one-dimensional SFTs (which all have computable languages), chain transitivity coincides with transitivity, and chain mixing with mixing.
This gives a simple characterization of generic limit sets among transitive SFTs.
By the results of \cite{Ka08}, chain transitivity and chain mixing of a given sofic shift are decidable in polynomial time, so generic limit sets form a well-behaved subclass of chain transitive sofic shifts as well.

\begin{corollary}
  \label{cor:transitive-SFTs}
  A one-dimensional transitive SFT is a generic limit set of some CA if and only if it is mixing.
\end{corollary}

Finally, we can completely characterize the generic limit sets among minimal subshifts.

\begin{corollary}
  \label{cor:minimal-characterization}
  A one-dimensional shift-minimal subshift is the generic limit set of a CA if and only if it is chain mixing and $\Delta^0_2$.
\end{corollary}

\begin{proof}
  By \cref{cor:minimal-delta02}, a shift-minimal generic limit set must be $\Delta^0_2$.
  Thus we may restrict our attention to $\Delta^0_2$ minimal subshifts.
  All minimal subshifts are in particular chain transitive.
  The result now follows from \cref{cor:ChainTransCharacterization}.
\end{proof}

%
%

\section{Future work}
\label{sec:future}

In this paper we have obtained several constraints on the structure and complexity of generic limit sets of cellular automata, and on the other hand proved that many of the complexity bounds are optimal.
In \cref{tab:limitsets}, inspired by the table in \cite[Section 7]{boyer2015mu}, we recapitulate the state of the art so far regarding several properties of the three commonly-considered CA attractors: the limit set, the $\mu$-limit set (where $\mu$ is the uniform Bernoulli measure) and the generic limit set.

\begin{table}[htp]
  \caption{Comparison of computability properties of different variants of limit sets.}
  \label{tab:limitsets}
  \begin{center}
    \begin{tabular}{>{\centering\arraybackslash}m{2.5cm}||>{\centering\arraybackslash}m{2.5cm}|>{\centering\arraybackslash}m{2.5cm}|>{\centering\arraybackslash}m{2.5cm}}
      \textbf{Problem or property} & \textbf{Limit set} & \textbf{$\mu$-limit set} & \textbf{Generic limit set}\\
      \hline
      \hline
      \textit{Being a singleton} & $\Sigma^0_1$-complete \cite{kari1992nilpotency} & $\Pi^0_3$-complete \cite[Theorem~5.7]{boyer2015mu} & $\Sigma^0_2$-complete \cite{To21}\\
      \hline
      \textit{Any non-trivial property} & $\Sigma^0_1$-hard \cite{kari} & $\Pi^0_3$-hard \cite[Theorem~5.2]{boyer2015mu} & undecidable \cite{De21}\\
      \hline
      \textit{Worst-case language} & $\Pi^0_1$-complete \cite[Theorem~4]{hurd} & $\Sigma^0_3$-complete \cite[Theorem~4.4]{boyer2015mu} & $\Sigma^0_3$-complete \cite[Theorem~1]{Ilkka}\\
      \hline
      \textit{Worst-case language when $f$ restricts to a shift} & computable (SFT) \cite[Theorem~1]{Ta07} & $\Sigma^0_3$-complete \cite[Theorem~4.4]{boyer2015mu} & $\Pi^0_2$-complete: \cref{cor:ShiftGivesMinimal}, \cref{pi2b} and~\cref{thm:MixingRealization}\\
      \hline
      \textit{Worst-case language when $f$ has equicontinuity points} & $\Pi^0_1$-complete \cite[Theorem 4]{hurd} & $\Sigma^0_1$ \cite[Theorem~4.2]{boyer2015mu} & $\Sigma^0_1$-complete: \cref{s1b} and~\cref{equiident}\\
    \end{tabular}
  \end{center}
\end{table}

The $\Pi^0_1$-completeness results on limit sets come from \cite[Theorem~4]{hurd} and its proof.
Hurd constructs a CA that simulates copies of a Turing machine on disjoint tapes and has a $\Pi^0_1$-complete limit language.
The tapes cannot be extended or destroyed, and no information can pass from one tape to another, so a short tape bordered by two other tapes forms a blocking word.
Hence this CA admits equicontinuity points.
The $\Sigma^0_3$-completeness results on $\mu$-limit sets follow from \cite[Theorem~4.4]{boyer2015mu}, where the authors constuct a CA that has a $\Sigma^0_3$-complete $\mu$-limit set on which it acts as the identity.

In \cref{thm:MixingRealization} we have studied the class of subshifts $X$ for which some CA $f$ satisfies $\gls(f) = X$ and $f|_X = \sigma|_X$.
There remains a gap between the upper bound of all chain mixing $\Pi^0_2$ subshifts, and the two incomparable lower bounds of those that also contain a nonempty $\Pi^0_1$ subshift, and those that are $\Delta^0_2$.
We believe that the class should have a relatively simple characterization, but dare not explicitly conjecture that the upper bound is strict.

\begin{ask}
  Which subshifts occur as the generic limit set of a CA that acts as the shift map on it?
\end{ask}

Of course, the same can be asked about the class of those subshifts $X$ for which some CA $f$ satisfies $\gls(f) = X$ and $f|_X = \mathrm{id}|_X$.
\Cref{equiident} shows that the complexity bound of $\Sigma^0_1$ for this class is optimal, but we did not investigate in details the structural properties of its elements.
From \cite[Prop. 3]{Ilkka} we know that they are at least topologically mixing, and from the proof of that result one can deduce that they satisfy a stronger mixing property.

\begin{ask}
  Which subshifts occur as the generic limit set of a CA that acts as the identity on it?
\end{ask}

\Cref{cor:transitive-SFTs} gives a simple characterization of the transitive SFTs that occur as generic limit sets.
We do not know how it would generalize to the class of all SFTs, or if the realizability of a given SFT or sofic shift as a generic limit set is even a decidable property.

\begin{ask}
  Which one-dimensional SFTs occur as generic limit sets of CA?
\end{ask}

\begin{ask}
  Which one-dimensional sofic shifts occur as generic limit sets of CA?
\end{ask}

\section*{Acknowledgements}

The first two authors would like to thank Mathieu Sablik for his wise, kind and steady help and guidance.

\bibliographystyle{alpha}
\bibliography{biblio}

\end{document}